\documentclass[11pt,a4paper]{amsart}
\usepackage{amscd}
\usepackage{comment}
\usepackage{oldgerm}
\usepackage{amsmath}
\usepackage{amssymb}
\usepackage{amsthm}
\usepackage[svgnames]{xcolor}
\usepackage{bm}
\usepackage{xr-hyper}
\usepackage[colorlinks,linktocpage=true,citecolor=blue,linkcolor=blue,urlcolor=blue]{hyperref}
\usepackage{graphicx}

\usepackage[all]{xy}
\usepackage[foot]{amsaddr}

\makeatletter
\@addtoreset{equation}{section}
\makeatother

\theoremstyle{definition}
\newtheorem{defn}[equation]{Definition}
\newtheorem{defn'}[equation]{Definition'}
\theoremstyle{plain}
\newtheorem{thm}[equation]{Theorem}

\newtheorem{prop}[equation]{Proposition}
\newtheorem{fact}[equation]{Fact}

\newtheorem{lem}[equation]{Lemma}

\newtheorem{conj}[equation]{Conjecture}
\theoremstyle{remark}
\newtheorem{rem}[equation]{Remark}
\newtheorem{ex}[equation]{Example}

\newcommand{\Z}{\mathbb{Z}}

\newcommand{\R}{\mathbb{R}}

\newcommand{\bS}{\mathbb{S}}

\newcommand{\pt}{\mathrm{pt}}

\newcommand{\del}{\partial}

\newcommand{\dR}{\mathrm{dR}}
\newcommand{\Bord}{\mathrm{Bord}}
\newcommand{\cw}{\mathrm{cw}}
\newcommand{\hBord}{h\Bord^{G_\nabla}}
\newcommand{\Ori}{\mathrm{Ori}}
\newcommand{\HS}{\mathrm{HS}}

\newcommand{\Thom}{\mathrm{Thom}}

\usepackage[shortlabels]{enumitem}

\DeclareMathOperator{\Hom}{Hom}

\def\beq#1\eeq{\begin{align}#1\end{align}}

\makeatletter
\newcommand*{\addFileDependency}[1]{
  \typeout{(#1)}
  \@addtofilelist{#1}
  \IfFileExists{#1}{}{\typeout{No file #1.}}
}
\makeatother

\begin{document}

\title[Differential models for $I\Omega^G$, II]{Differential models for the Anderson dual to bordism theories and invertible QFT's, II}

\author[M. Yamashita]{Mayuko Yamashita}
\address{Department of Mathematics, Kyoto University, 
Kita-shirakawa Oiwake-cho, Sakyo-ku, Kyoto, 606-8502, Japan
}
\email{yamashita.mayuko.2n@kyoto-u.ac.jp}
\subjclass[]{}
\maketitle

\begin{abstract}
    This is the second part of the work on differential models of the Anderson duals to the stable tangential $G$-bordism theories $I\Omega^G$, motivated by classifications of invertible QFT's. 
    Using the model constructed in the first part \cite{YamashitaYonekura2021}, in this paper we show that pushforwards in generalized differential cohomology theories
    induces transformations between differential cohomology theories which refine the Anderson duals to multiplicative genera. 
    This gives us a unified understanding of an important class of elements in the Anderson duals with physical origins. 
\end{abstract}

\tableofcontents

\section{Introduction}
This is the second part of the work on differential models of the Anderson duals to the stable tangential $G$-bordism theories $I\Omega^G$ motivated by classifications of invertible QFT's. 
The generalized cohomology theory $I\Omega^G$ is conjectured by Freed and Hopkins \cite[Conjecture 8.37]{Freed:2016rqq} to classify deformation classes of possibly non-topological invertible 
quantum field theories (QFT's) on stable tangential $G$-manifolds. 
Motivated by this conjecture, in the previous paper \cite{YamashitaYonekura2021} by Yonekura and the author, we constructed a model $I\Omega^G_\dR$ of $I\Omega^G$ and its differential extension $\widehat{I\Omega^G_\dR}$ by abstractizing properties of partition functions for invertible QFT's. This paper is devoted to their relations with {\it multiplicative genera}. 
We show that {\it pushforwards} (also called {\it integrations}) in generalized differential cohomology theories allow us to construct differential refinements of certain cohomology transformations which arise from the Anderson dual to multiplicative genera and the module structures of the Anderson duals. 
This gives us a unified understanding of an important class of elements in the Anderson duals with physical origins. 
Moreover, having transformations in the differential level gives a more direct connection to physical picture, since they can be regarded as transformations of QFTs without taking deformation classes, recovering the information of partition functions. 

First, we explain the motivations of the previous paper. 
As we recall in Section \ref{sec_preliminary_2}, the differential group $(\widehat{I\Omega^G_\dR})^n(X)$ consists of pairs $(\omega, h)$, where $\omega \in \Omega^n_{\mathrm{clo}}(X; {H}^\bullet(MTG; \R))$ where $n$ is the total degree, and $h$ is a map which assigns $\R/\Z$-values to {\it differential stable tangential $G$-cycles} of dimension $(n-1)$ over $X$, which satisfy a compatibility condition with respect to bordisms. 
The physical interpretation is that $h$ is the complex phase of the partition function of an invertible QFT. 
For example, given a hermitian line bundle with unitary connection over $X$, the pair of the first Chern form and the holonomy function gives an element for $G = \mathrm{SO}$ and $n = 2$ (Subsection \ref{subsec_hol_target}). 
Similarly, given a hermitian vector bundle with unitary connection, we can construct even-degree elements for $G = \mathrm{Spin}^c$ using the reduced eta invariants of twisted Dirac operators (Subsection \ref{subsec_complex_eta}): this theory is related to anomaly of spinor field theory. 
Then, a natural mathematical question arises: {\it what are these elements mathematically?}
It is natural to expect a topological characterization of these elements. 
Questions of this kind also appears in \cite[Conjecture 9.70]{Freed:2016rqq} (also recalled in Conjecture \ref{conj_anomaly_fermion} below). 
This paper is devoted to this question. 
Actually, these examples are special cases of the general construction in this paper which we now explain. 

Now we explain the general settings. 
In this paper, the tangential structure groups $G = \{G_d, s_d, \rho_d\}_{d \in \Z_{\ge 0}}$ (see Section \ref{sec_preliminary_2}) is assumed to be {\it multiplicative}, i.e., the corresponding Madsen-Tillmann spectrum $MTG$ is equipped with a structure of a ring spectrum. 
Assume we are given a ring spectrum $E$ with a homomorphism of ring spectra, 
\begin{align*}
    \mathcal{G} \colon MTG \to E. 
\end{align*}
such $\mathcal{G}$ is also called a {\it multiplicative genus}, and examples include the usual orientation $\tau \colon MT\mathrm{SO} \to H\Z$ and the Atiyah-Bott-Shapiro orientations $\mathrm{ABS} \colon MT\mathrm{Spin}^c \to K$ and $\mathrm{ABS} \colon MT\mathrm{Spin}\to KO$. 

On the topological level, a ring homomorphism $\mathcal{G} \colon MTG \to E$ gives {\it pushforwards} in $E$ for proper $G$-oriented smooth maps. 
{\it Pushforwards in differential cohomology}, or {\it differential pushforwards}, are certain differential refinements of topological pushforwards. 
Basically, they consist of corresponding maps in $\widehat{E}$ for each proper map with a ``differential $E$-orientation''. 
The formulations depend on the context. 
To clarify this point, in this paper we use the differential extension $\widehat{E}_\HS$ of $E$ constructed by Hopkins-Singer \cite{HopkinsSinger2005}, and use the formulation of differential pushforwards in that paper. 
Throughout this paper, we assume that $E$ is {\it rationally even}, i.e., $E^{2k+1}(\pt) \otimes \R=0$ for any integer $k$. 
In this case, by \cite{Upmeier2015} there exists a canonical multiplicative structure on the Hopkins-Singer's differential extension $\widehat{E}_{\mathrm{HS}}^*(-; \iota_E)$ associated to a fundamental cycle $\iota_E \in Z^0(E; V_E^\bullet)$. 
The theory of differential pushforwards gets simple in this case.
This point is explained in Subsection \ref{subsec_HS_push} and Appendix \ref{app_diff_push}.
Of course our result applies to any model of differential extension $\widehat{E}$ of $E$ which is isomorphic to the Hopkins-Singer's model. 
Practically, most known examples of differential extensions are isomorphic to Hopkins-Singer's model (see Footnote \ref{footnote_uniqueness}).  
The holonomy functions are examples of differential pushforwards in the case $\tau \colon MT\mathrm{SO} \to H\Z$, and the reduced eta invariants are those for $\mathrm{ABS} \colon MT\mathrm{Spin}^c \to K$ by the result of Freed and Lott \cite{FL2010} and Klonoff \cite{KlonoffDiffK}.

Let $n$ be an integer such that $E^{1-n}(\pt) \otimes \R = 0$.  
As we show in Subsection \ref{subsec_genus_main}, the above data defines the following natural transformation, 
\begin{align}\label{eq_goal_theory}
    \Phi_{\mathcal{G}} \colon \widehat{E}_{\mathrm{HS}}^*(-; \iota_E) \otimes IE^n(\pt) \to (\widehat{I\Omega^G_{\mathrm{dR}}})^{*+n}(-), 
\end{align}
on $\mathrm{Mfd}^{\mathrm{op}}$ (Definition \ref{def_theory}).

The main result of this paper is the following topological characterization of the transformation \eqref{eq_goal_theory}. 
\begin{thm}[{=Theorem \ref{thm_genus_theory_2}}]\label{thm_genus_theory}
In the above settings, 
let $X$ be a manifold and $k$ be an integer. 
For $\widehat{e} \in \widehat{E}_\HS^k(X; \iota_E)$ and $\beta \in IE^n(\pt)$, the element $I(\Phi_{\mathcal{G}}(\widehat{e} \otimes \beta)) \in (I\Omega^G)^{k+n}(X) = [X^+\wedge MTG ,  \Sigma^{k+n}I\Z]$ coincides with the following composition, 
\begin{align}\label{eq_genus_theory}
    X^+\wedge MTG \xrightarrow{e \wedge \mathcal{G}} 
    \Sigma^k E \wedge E \xrightarrow{\mathrm{multi}}
    \Sigma^k E \xrightarrow{\beta} \Sigma^{k+n}I\Z. 
\end{align}
Here we denoted $e := I(\widehat{e}) \in E^k(X)$. 
\end{thm}

As we will see in Subsection \ref{subsec_genus_example}, this result applies to the above mentioned examples as follows. 
\begin{ex}[Subsection \ref{subsec_hol_target}]
Set $E = H\Z$ with $\tau \colon MT\mathrm{SO} \to H\Z$. We have the Anderson self-duality element $\gamma_{H\Z} \in IH\Z^0(\pt)$. 
The transformation
\begin{align}
     \Phi_{\tau}(-\otimes \gamma_H) \colon \widehat{H}^2(X; \Z) \to (\widehat{I\Omega^{\mathrm{SO}}_{\mathrm{dR}}})^{2}(X)
\end{align}
sends the class of a hermitian line bundle with connection $[L, \nabla] \in \widehat{H}^2(X; \Z)$ to the element $(c_1(\nabla), \mathrm{Hol}_{\nabla}) \in (\widehat{I\Omega^{\mathrm{SO}}_{\mathrm{dR}}})^{2}(X)$. 
Applying Theorem \ref{thm_genus_theory}, we see that its deformation class coincides with the following composition, 
\begin{align*}
    X^+ \wedge MT\mathrm{SO} \xrightarrow{ c_1(L) \wedge \tau} 
    \Sigma^2 H\Z\wedge H\Z  \xrightarrow{\mathrm{multi}}
    \Sigma^2 H\Z \xrightarrow{\gamma_H} \Sigma^{2}I\Z. 
\end{align*}
\end{ex}

\begin{ex}[Subsections \ref{subsec_complex_eta} and \ref{subsec_anomaly}]
Set $E=K$ with the Atiyah-Bott-Shapiro orientations $\mathrm{ABS} \colon MT\mathrm{Spin}^c \to K$. 
We have the Anderson self-duality element $\gamma_K\in IK^0(\pt)$. The transformation
\begin{align}
    \Phi_{\mathrm{ABS}}(-\otimes \gamma_K) \colon \widehat{K}^{2k}(X) \to (\widehat{I\Omega^{\mathrm{Spin}^c}_{\mathrm{dR}}})^{2k}(X). 
\end{align}
maps the class $[W, h^W, \nabla^W, 0] \in \widehat{K}^0(X) \simeq \widehat{K}^{2k}(X)$ of hermitian vector bundle with unitary connection to the element
$ \left((\mathrm{Ch}(\nabla^W) \otimes \mathrm{Todd})|_{2k}, \overline{\eta}_{\nabla^W}\right) \in \left(\widehat{I\Omega^{\mathrm{Spin}^c}}\right)^{2k}(X)$. Applying Theorem \ref{thm_genus_theory}, we see that its deformation class coincides with the following composition, 
\begin{align*}
    X^+ \wedge MT\mathrm{Spin}^c \xrightarrow{[E]\wedge \mathrm{ABS}} 
    K\wedge K \xrightarrow{\mathrm{multi}}
    K \xrightarrow[\simeq]{\mathrm{Bott}} \Sigma^{2k}K \xrightarrow{\gamma_K} \Sigma^{2k}I\Z. 
\end{align*}

In Subsection \ref{subsec_anomaly}, we explain that this transformation can be interpreted as taking anomaly theories of complex spinor field theories. 
\end{ex}

The paper is organized as follows. 
In Section \ref{sec_preliminary_2} we recall the definition of the differential models in \cite{YamashitaYonekura2021}. 
Section \ref{sec_main_genus} is the main part of this paper. 
We construct the natural transformation \eqref{eq_goal_theory} and prove Theorem \ref{thm_genus_theory} in Subsection \ref{subsec_genus_main}. 
We explain some examples, as well as relation with anomalies, in Subsection \ref{subsec_genus_example}. 
As we explain in Subsection \ref{subsec_HS_push}, there are certain subtleties regarding the formulations of differential pushforwards. 
In Appendix \ref{app_diff_push}, we collect the necessary results concerning differential pushforwards for submersions when $E$ is rationally even. 

\subsection{Notations and Conventions}

\begin{itemize}
\item By {\it manifolds}, we mean smooth manifolds with corners. 
We use the conventions explained in \cite[Subection 2.3]{YamashitaYonekura2021}. 
\item The space of $\R$-valued differential forms on a manifold $X$ is denoted by $\Omega^*(X)$. 
    \item We deal with differential forms with values in a graded real vector space $V^\bullet$. 
In the notation $\Omega^n(-; V^\bullet)$, $n$ means the {\it total} degree. 
In the case if $V^\bullet$ is infinite-dimensional, we topologize it as the colimit of all its finite-dimensional subspaces with the caonical topology, and set $\Omega^n(X; V^\bullet) := C^\infty(X; (\wedge T^*X \otimes_\R V^\bullet)^n)$. 
This means that, any element in $\Omega^n(X; V^\bullet)$ can locally be written as a finite sum $\sum_{i} \xi_i \otimes \phi_i$
with $\xi_i \in \Omega^{m_i}(X)$ and $\phi_i \in V^{n-m_i}$ for some $m_i$ for each $i$. 
The space of closed forms are denoted by $\Omega_{\mathrm{clo}}^n(-; V^\bullet)$. 
    \item For a manifold $X$ and a real vector space $V$, we denote by $\underline{V}$ the trivial bundle $\underline{V} := X \times V$ over $X$. 
\item For a topological space $X$, we denote by $p_X \colon X \to \pt$ the map to $\pt$. 
We set $X^+ := (X \sqcup \{*\}, \{*\})$. 
\item For two topological spaces $X$ and $Y$, we denote by $\mathrm{pr}_X \colon X \times Y \to X$ the projection to $X$. 
\item We set $I := [0, 1]$. 
\item For a real vector bundle $V$ over a topological space, we denote its orientation line bundle (rank-$1$ real vector bundle) by $\Ori(V)$. 
For a manifold $M$, we set $\Ori(M) := \Ori(TM)$. 
\item For a {\it spectrum} $\{E_n\}_{n \in \Z}$, we require the adjoints $E_n \to \Omega E_{n+1}$ of the structure homomorphisms are homeomorphisms. 
For a sequence of pointed spaces $\{E'_n\}_{n \in \Z_{\ge a}}$ with maps $\Sigma E'_n \to E'_{n+1}$, we define its {\it spectrification} $LE' :=\{(LE')_n\}_{n \in \Z}$ to be the spectrum given by
\begin{align*}
    (LE')_n := \varinjlim_k \Omega^k E'_{n+k}. 
\end{align*}
    
\item For a generalized cohomology theory $E$, we set $V_E^\bullet := E^\bullet(\pt) \otimes \R$. 

\item The {\it Chern-Dold homomorphism} \cite[Chapter II, 7.13]{rudyak1998} for a generalized cohomology theory $E$ is denoted by
\begin{align}\label{eq_chd_E}
    \mathrm{ch} \colon E^*(-) \to H^*(-; V_E^\bullet). 
\end{align}

\item We use the axiomatic framework of generalized differential cohomology given in \cite{BSDiffKSurvey} (also recalled in \cite[Subsection 2.2]{YamashitaYonekura2021}). 
We abuse the notations $R$, $a$ and $I$ for the structure maps for general differential cohomology theories, following the standard notations in those papers. 
\end{itemize}

\section{Preliminaries from \cite{YamashitaYonekura2021}}\label{sec_preliminary_2}
In this section we recall necessary parts of the previous paper \cite{YamashitaYonekura2021}. 
In \cite{YamashitaYonekura2021}, we constructed a model $\widehat{I\Omega^G_{\dR}}$ of a differential extension of the Anderson dual to the tangential $G$-bordism homology theory $I\Omega^G$. 

First we recall the definition of the {\it Anderson duals} to spectra (see \cite[Section 2.1]{YamashitaYonekura2021}, \cite[Appendix B]{HopkinsSinger2005} and \cite[Appendix B]{FMS07}). 
The functor $X \mapsto \mathrm{Hom}(\pi_*(X), \R/\Z)$ on the stable homotopy category is represented by a spectrum denoted by $I(\R/\Z)$.
The {\it Anderson dual to the sphere spectrum}, denoted by $I\Z$, is defined as the homotopy fiber of the morphism $H\R \to I(\R/\Z)$ representing the transformation $\mathrm{Hom}(\pi_*(-), \R)  \to \mathrm{Hom}(\pi_*(-), \R/\Z)$. 
For any spectra $E$, the {\it Anderson dual to $E$}, denoted by $IE$, is defined to be the function spectrum from $E$ to $I\Z$, $IE := F(E, I\Z)$. 
This implies that we have the following exact sequence for any spectra $X$. 
\begin{align}\label{eq_exact_IE_part2}
    \cdots \to \mathrm{Hom}(E_{n-1}(X), \R) &\xrightarrow{\pi} \Hom(E_{n-1}(X), \R/\Z) \to IE^n(X) \\ & \to \mathrm{Hom}(E_{n}(X), \R) \xrightarrow{\pi} \Hom(E_n(X), \R/\Z) \to \cdots \ (\mbox{exact}). \notag
\end{align}

In \cite{YamashitaYonekura2021} and the current paper, we are particularly interested in the Anderson dual to the $G$-bordism homology theory. 
Here, 
$G = \{G_d, s_d, \rho_d\}_{d \in \Z_{\ge 0}}$ is a sequence of compact Lie groups equipped with homomorphisms $s_d \colon G_d \to G_{d+1}$ and $\rho_d \colon G_d \to \mathrm{O}(d, \R)$ for each $d$, such that the following diagram commutes. 
 \begin{align*}
     \xymatrix{
 G_d \ar[r]^-{\rho_d} \ar[d]^{s_d} & \mathrm{O}(d, \R) \ar[d] \\
 G_{d+1} \ar[r]^-{\rho_{d+1}} & \mathrm{O}(d+1, \R)  
 }. 
 \end{align*}
Here we use the inclusion $\mathrm{O}(d, \R) \hookrightarrow \mathrm{O}(d+1, \R) $ defined by 
\begin{align*}
     A \mapsto \left[
     \begin{array}{c|c}
     1 & 0  \\ \hline
     0 &A 
     \end{array}
     \right]  
 \end{align*}
throughout this paper. 
Given such $G$, the stable tangential $G$-bordism homology theory assigns the stable tangential bordism group $(\Omega^G)_*(-)$. 
It is represented by the {\it Madsen-Tillmann spectrum} $MTG$, which is a variant of the Thom spectrum $MG$. 
For details see for example \cite[Section 6.6]{Freed19}. 
In this paper we take $MTG$ and $MG$ to be a spectrum as in \cite[(4.60)]{HopkinsSinger2005}. 
The Anderson dual $I\Omega^G$ fits into the following exact sequence. 
\begin{align}\label{eq_exact_IOmega_Part2}
    \cdots \to \mathrm{Hom}(\Omega^G_{n-1}(X), \R) &\to \Hom(\Omega^G_{n-1}(X), \R/\Z) \to (I\Omega^G)^n(X) \\ & \to \mathrm{Hom}(\Omega^G_{n}(X), \R) \to \Hom(\Omega^G_n(X), \R/\Z) \to \cdots \ (\mbox{exact}). \notag
\end{align}

Now we proceed to recall the definition of $\widehat{I\Omega^G_\dR}$. 
In \cite{YamashitaYonekura2021}, we defined the relative groups $\widehat{I\Omega^G_\dR}(X, Y)$. 
But since we only deal with the absolute case $X = (X, \varnothing)$ in this paper, we concentrate on this case. 
To recall the definition of $\widehat{I\Omega^G_\dR}$, we first recall the {\it differential stable $G$-structures} on vector bundles. 

\begin{defn}[{Differential stable $G$-structures on vector bundles, \cite[Definition 3.1]{YamashitaYonekura2021}}]\label{def_diff_Gstr_vec_part2}
Let $V$ be a real vector bundle of rank $n$ over a manifold $M$. 
\begin{enumerate}
    \item 
A {\it representative of differential stable $G$-structure} on $V$ is a quadruple $\widetilde{g} = (d, P, \nabla, \psi)$, where $d \ge n$ is an integer, 
$(P, \nabla)$ is a principal $G_d$-bundle with connection over $M$ and
$\psi \colon P \times_{\rho_d} \R^d \simeq  \underline{\R}^{d-n} \oplus V $ is an isomorphism of vector bundles over $M$. 
\item We define the {\it stabilization} of such $\widetilde{g}$ by $\widetilde{g}(1) := (d+1, P(1) := P \times_{s_{d}}G_{d+1}, \nabla(1), \psi(1))$, where $\nabla(1)$ and $\psi(1)$ are naturally induced on $P(1)$ from $\nabla$ and $\psi$, respectively. 
\item A {\it differential stable $G$-structure} $g$ on $V$ is a class of representatives $\widetilde{g}$ under the relation $\widetilde{g} \sim_{\mathrm{stab}} \widetilde{g}(1)$. 
\item Suppose we have two representatives of the forms $\widetilde{g} = (d, P, \nabla, \psi)$ and $\widetilde{g} = (d, P, \nabla, \psi')$, such that $\psi$ and $\psi'$ are homotopic. 
In this case, the resulting differential stable $G$-structures $g$ and $g'$ are called {\it homotopic}. 
\end{enumerate}
\end{defn}

If we forget the information of the connection $\nabla$, we get the corresponding notion of {\it (topological) differential stable $G$-structures}. 
For a differential stable $G$-structure $g$, we denote the underlying topological structure by $g^{\mathrm{top}}$. 
Similar remarks apply to the various definitions below. 

We also recall the normal variant, which we use in Appendix \ref{app_diff_push}. 

\begin{defn}[{Differential stable normal $G$-structures on vector bundles, \cite[Definition 4.75]{YamashitaYonekura2021}}]\label{def_diff_Gstr_vec_normal_part2}
Let $V$ be a real vector bundle of rank $n$ over a manifold $M$. 
\begin{enumerate}
    \item[(1)]
A {\it representative of differential stable normal $G$-structure} on $V$ is a quadruple $\widetilde{g}^\perp = (d, P, \nabla, \psi)$, where $d \ge n$ is an integer, 
$(P, \nabla)$ is a principal $G_{d-n}$-bundle with connection over $M$ and
$\psi \colon (P \times_{\rho_{d-n}} \R^{d-n}) \oplus V \simeq  \underline{\R}^{d}  $ is an isomorphism of vector bundles over $M$. 
\item[(2), (3)] We define the {\it stabilization} of such $\widetilde{g}^\perp$ in the same way as Definition \ref{def_diff_Gstr_vec_part2}, and a {\it differential stable normal $G$-structure} $g^\perp$ on $V$ is defined to be a class of representatives under the stabilization relation. 
\item[(4)] We define the {\it homotopy} relation between two such $g^\perp$'s also in the same way. 
\end{enumerate}
\end{defn}

A {\it differential stable tangential $G$-structure} on a manifold $M$ is a differential stable $G$-structure on $TM$. 
Given a manifold $X$, a {\it differential stable tangential $G$-cycle} of dimension $n$ over $X$ is a triple $(M, g, f)$, where $M$ is an $n$-dimensional closed manifold, $g$ is a differential stable tangential $G$-structure on $M$ and $f \in C^\infty(M, X)$. 
Using these cycles, we defined
\begin{itemize}
    \item The abelian group $\mathcal{C}^{G_\nabla}_n(X)$, consisting of the equivalence classes $[M, g, f]$ of differential stable tangential $G$-cycles of dimension $n$ over $X$, with the equivalence relation generated by isomorphisms, opposite relation and homotopy relation \cite[Definition 3.5]{YamashitaYonekura2021}. 
    \item The Picard groupoid $\hBord_{n}(X)$, whose objects are differential stable tangential $G$-cycles $(M, g, f)$ of dimension $n$ over $X$, and morphisms are the bordism classes $[W,g_W, f_W]$ of bordisms $(W,g_W, f_W) \colon (M_-, g_-, f_-) \to (M_+, g_+, f_+)$ \cite[Definition 3.8]{YamashitaYonekura2021}. 
\end{itemize}
By the theorem of Pontryagin-Thom we have an equivalence of Picard groupoids \cite[Lemma 3.10]{YamashitaYonekura2021}, 
\begin{align}\label{eq_lem_cat_equivalence_part2}
    h\Bord^{G_\nabla}_n(X) \simeq \pi_{\le 1}(L(X^+\wedge MTG)_{-n}).  
\end{align}
Here the right hand side denotes the fundamental Picard groupoid. 

Let $\R_{G_d}$ denote the $G_d$ module with the underlying vector space $\R$ and the $G_d$-action by $\mathrm{det} \circ \rho_d \colon G \to \{\pm 1\}$.  
By the Thom isomorphism, we have 
\begin{align*}
    N_G^\bullet &:= H^\bullet(MTG; \R) \simeq  \varprojlim_{d}(\mathrm{Sym}^{\bullet/2}\mathfrak{g}_d^* \otimes_\R \R_{G_d})^{G_d} \simeq H^\bullet(MG; \R) =: N_{G^\perp}^\bullet. 
\end{align*}
The fourth arrow in \eqref{eq_exact_IOmega_Part2} gives the homomorphism
\begin{align}\label{eq_ch_N_part2}
    \mathrm{ch}' \colon (I\Omega^G)^*(X) \to H^*(X; N_G^\bullet) \simeq \Hom(\Omega^G_*(X), \R). 
\end{align}
By \cite[Proposition 4.9]{YamashitaYonekura2021} we have a canonical homomorphism
\begin{align}\label{eq_exact_V_N_part2}
    q\colon V_{I\Omega^G}^n \xrightarrow{q} N_G^n . 
\end{align}
The map $q$ is isomorphism if $\Omega^G_n(\pt)$ is finitely generated for all $n$. 

Let $\mathcal{V}^*$ be any $\Z$-graded vector space over $\R$. 
Given a vector bundle $V \to M$ with a differential stable $G$-structure $g$, we get a homomorphism of $\Z$-graded real vector spaces by the Chern-Weil construction \cite[Definition 4.4 and Remark 4.10]{YamashitaYonekura2021}, 
\begin{align}\label{eq_def_cw_hom_part2}
    \mathrm{cw}_g  \colon \Omega^*(M; H^\bullet(MTG; \mathcal{V}^*)) = \Omega^*\left(M; H^\bullet(MG; \mathcal{V}^*)\right) \to \Omega^*(M; \mathrm{Ori}(V)\otimes_\R  \mathcal{V}^*). 
\end{align}
Since the orientation bundle of a vector bundle and its normal bundle are canonically identified, for $V \to M$ equipped with a differential stable normal $G$-structure $g^\perp$, we also get
\begin{align}\label{eq_def_cw_hom_normal_part2}
    \mathrm{cw}_{g^\perp} \colon \Omega^*(M; H^\bullet(MTG; \mathcal{V}^*)) = \Omega^*\left(M; H^\bullet(MG; \mathcal{V}^*)\right) \to \Omega^*(M; \mathrm{Ori}(V)\otimes_\R  \mathcal{V}^*). 
\end{align}

For $\omega\in \Omega_{\mathrm{clo}}^n(X; H^\bullet(MTG; \mathcal{V}^*))$ we get a homomorphism
\begin{align}\label{eq_def_cw_omega_mor_part2}
    \mathrm{cw}(\omega) \colon \Hom_{h\Bord^{G_\nabla}_{N-1}(X)}((M_-, g_-, f_-), (M_+, g_+, f_+)) &\to \mathcal{V}^{n-N} \\
    [W, g_W, f_W] &\mapsto \int_W \mathrm{cw}_{g_W}((f_W)^*\omega),  \notag
\end{align}
for each pair of objects $(M_\pm, g_\pm, f_\pm)$ in $h\Bord^{G_\nabla}_{N-1}(X)$. 
In the following definition, we use \eqref{eq_def_cw_omega_mor_part2} in the case $\mathcal{V}^* = \R$ and $n = N$. 

\begin{defn}[{$(\widehat{I\Omega^G_{\mathrm{dR}}})^*$ and $(I\Omega^G_{\mathrm{dR}})^*$, \cite[Definition 4.15]{YamashitaYonekura2021}}]
Let $X$ be a manifold and $n \in \Z_{\ge 0}$. 
\begin{enumerate}
    \item 
Define $(\widehat{I\Omega^G_{\mathrm{dR}}})^n(X)$ to be an abelian group consisting of pairs $(\omega, h)$, such that
\begin{enumerate}
    \item $\omega$ is a closed $n$-form\footnote{Recall that $n$ is the {\it total} degree. } $\omega\in \Omega_{\mathrm{clo}}^n(X; N_{G}^\bullet)$. 
    \item $h$ is a group homomorphism
    $h \colon \mathcal{C}^{G_\nabla}_{n-1}(X) \to \R/\Z$. 
\item $\omega$ and $h$ satisfy the following compatibility condition. 
Assume that we are given two objects $(M_-, g_-, f_-)$ and $(M_+, g_+, f_+)$ in $h\Bord^{G_\nabla}_{n-1}(X)$ and a morphism $[W, g_W, f_W]$ from the former to the latter. 
Then we have
\begin{align*}
    h([M_+, g_+, f_+]) - h([M_-, g_-, f_-]) = \mathrm{cw}(\omega)([W, g_W, f_W]) \pmod \Z.
\end{align*}
\end{enumerate}
Abelian group structure on $(\widehat{I\Omega^G_{\mathrm{dR}}})^n(X)$ is defined in the obvious way. 

\item
We define a homomorphsim of abelian groups, 
\begin{align*}
    a \colon \Omega^{n-1}(X; N_{G}^\bullet)/\mathrm{Im}(d) &\to  (\widehat{I\Omega^G_{\mathrm{dR}}})^n(X) \\
    \alpha &\mapsto (d\alpha, \mathrm{cw}(\alpha)). \notag
\end{align*}
Here the homomorphism $\mathrm{cw}(\alpha) \colon \mathcal{C}^{G_\nabla}_{ n-1}(X) \to \R/\Z$ is defined by
\begin{align*}
    \mathrm{cw}(\alpha) ([M, g, f]) := \int_M \mathrm{cw}_g(f^*\alpha) \pmod \Z. 
\end{align*}
We set
\begin{align*}
    (I\Omega^G_{\mathrm{dR}})^n(X) := (\widehat{I\Omega^G_{\mathrm{dR}}})^n(X)/ \mathrm{Im}(a). 
\end{align*}
\end{enumerate}
For $n \in \Z_{< 0}$ we set $(\widehat{I\Omega^G_{\mathrm{dR}}})^n(X) = 0$ and $(I\Omega^G_{\mathrm{dR}})^n(X) = 0$. 
\end{defn}

We defined the structure homomorphisms $R$, $a$ and $I$ along with the $S^1$-integration map $\int$ for $\widehat{I\Omega^G_\dR}$. 
One of the main results of \cite{YamashitaYonekura2021} is the following. 
\begin{thm}[{\cite[Theorem 4.51]{YamashitaYonekura2021}}]
We have a natural isomorphism of functors $\mathrm{Mfd} \to \mathrm{Ab}^\Z$, 
\begin{align*}
    (I\Omega^G_{\mathrm{dR}})^* \simeq (I\Omega^G)^*
\end{align*}
Moreover, the functor $\widehat{I\Omega^G_{\mathrm{dR}}}$, along with the structure maps introduced in \cite{YamashitaYonekura2021}, is a differential extension with $S^1$-integration of the pair $\left((I\Omega^G)^*, \mathrm{ch}'\right)$, where $\mathrm{ch}'$ is defined in \eqref{eq_ch_N_part2}. 
\end{thm}

\section{Pushforwards in differential cohomologies and the Anderson duality}\label{sec_main_genus}

This is the main section of this paper. 
The main part is Subsection \ref{subsec_genus_main}, where we construct the natural transformation \eqref{eq_goal_theory} and prove Theorem \ref{thm_genus_theory}. 
Subsections \ref{subsec_HS_push} and \ref{subsec_genus_functor} are preparation for the construction and proof. 
We explain some examples, as well as relations with anomaly, in Subsection \ref{subsec_genus_example}. 

\subsection{Preliminary--Differential pushforwards in the Hopkins-Singer model}\label{subsec_HS_push}
In this subsection, we briefly explain the differential extensions of generalized cohomology theories constructed by Hopkins-Singer and the differential pushforwards (called {\it integration} in \cite{HopkinsSinger2005}) in that model. 
We explain it in more detail in Appendix \ref{app_diff_push}.  

On the topological level, a ring homomorphism $\mathcal{G} \colon MTG \to E$ gives {\it pushforwards} in $E$ for $G$-oriented proper smooth maps. 
For proper smooth maps $p \colon N \to X$ of relative dimension $r := \dim N - \dim X$ with (topological) stable relative tangential $G$-structures $g_{p}^{\mathrm{top}}$, we get the corresponding pushforward map, 
\begin{align}\label{eq_top_push}
    (p, {g_{p}^{\mathrm{top}}})_* \colon E^*(N) \to E^{* -r}(X). 
\end{align}
In particular in the case $X = \pt$, for a closed manifold $M$ of dimension $n$ with a stable tangential $G$-structure $g^{\mathrm{top}}$ (\cite[Definition 3.2]{YamashitaYonekura2021}), we get
\begin{align*}
    (p_M, g^{\mathrm{top}})_* \colon E^*(M) \to E^{*-n}(\pt). 
\end{align*}

There are notions of differential refinements of the pushforward maps in $\widehat{E}$. 
For example see \cite[Section 4.10]{HopkinsSinger2005}, \cite[Section 2]{BunkeSchickSchroderMU} and \cite[Section 4.8 -- 4.10]{Bunke2013}. 
Basically, they consist of corresponding maps in $\widehat{E}$ for each proper map with a ``differential $E$-orientation''. 
The formulations depend on the context. 
In this paper, we adopt the one by Hopkins-Singer\footnote{
In particular we use the differential extension $\widehat{E}_\HS$. 
Practically this is not restrictive. 
We are assuming $E$ is rationally even and multiplicative, so $\widehat{E}_\HS$ is equipped with a canonical multiplicative structure by \cite{Upmeier2015}. 
Thus, when the coefficients of $E$ are countably generated, we can apply the uniqueness result in \cite[Theorem 1.7]{BunkeSchick2010} to conclude that any other multiplicative differential extension (defined on the category of all smooth manifolds) is isomorphic to $\widehat{E}_\HS$. 
\label{footnote_uniqueness}
}. 

Hopkins and Singer gave a model of differential extensions, which we denote by $\widehat{E}_\HS^*(-; \iota_E)$, for any spectrum $E$, in terms of {\it differential function complexes}. 
In general we choose a $\Z$-graded vector space $V^\bullet$, and a singular cocycle $\iota_E \in Z^0(E; V^\bullet) = \varprojlim_{n}Z^n(E_{n}; V^\bullet)$. 
Then for each $n$ and for each manifold $X$, we get a simplicial complex called {\it differential function complex}, 
\begin{align*}
    (E_n; \iota_n)^X = (E; \iota)^X_n, 
\end{align*}
consisting of {\it differential functions} $X \times \Delta^\bullet \to (E_n; \iota_n)$. 
This complex has a filtration $\mathrm{filt}_{s}(E; \iota)^X_n$, $s \in \Z_{\ge 0}$. 
The differential cohomology group is defined as (it is denoted by $E(n)^n(X; \iota)$ in \cite{HopkinsSinger2005}),
\begin{align*}
    \widehat{E}^n_\HS(X; \iota) := \pi_0 \mathrm{filt}_0(E; \iota)^X_n. 
\end{align*}
In particular this means that an element in $\widehat{E}^n_\HS(X; \iota)$ is represented by a {\it differential function} $(c, h,  \omega) \colon X \to (E_n; \iota_n)$, consisting of a continuous map $c \colon X \to E_n$, a closed form $\omega \in \Omega_{\mathrm{clo}}^n(X; V^\bullet)$ and a singular cochain $h \in C^{n-1}(X; V^\bullet)$ such that $\delta h = c^* \iota_n - \omega$ as smooth singular cocycles. 

A particularly important case is when $V = V_E^\bullet$ and $\iota_E \in Z^0(E; V_E^\bullet)$ is the {\it fundamental cocycle}, i.e., a singular cocycle representing the Chern-Dold character of $E$. 
In this case the associated differential cohomology groups $\widehat{E}^n_\HS(X; \iota_E)$ satisfies the axioms of differential cohomology theory in \cite{BunkeSchick2010}. 
The isomorphism class of the resulting group is independent of the choice of the fundamental cocycle $\iota_E$, with an isomorphism given by a cochain cobounding the difference. 

In \cite[Section 4.10]{HopkinsSinger2005}, a differential pushforward is defined simply as maps of differential function spaces\footnote{
This point is important in the proof of Proposition \ref{prop_genus_functor}, which is the main ingredient of the proof of the main result (Theorem \ref{thm_genus_theory_2}). 
This is the reason why we want to use the Hopkins-Singer's formulation. 
\label{footnote_why_HS}
}, 
\begin{align}\label{eq_hat_mathcalG}
    \widehat{\mathcal{G}} \colon \left(MTG_{-r} \wedge (E_{n})^+; V_{\mathcal{G}}(\iota_{MTG})_{-r} \cup (\iota_E)_n \right) \to (E; \iota_E)_{n-r}, 
\end{align}
refining the map $MTG \wedge (E_{n})^+ \xrightarrow{\mathcal{G} \wedge \mathrm{id}}E \wedge (E_{n})^+ \xrightarrow{\mathrm{multi}} \Sigma^n E$. 
Here we are taking $V = V_E^\bullet$, and the cocycle $V_{\mathcal{G}}(\iota_{MTG}) \in Z^0(MTG; V_E^\bullet)$ is obtained by applying $V_{\mathcal{G}} \colon V_{MTG} \to V_E$ on the coefficient of $\iota_{MTG}$. 
Then\footnote{
As we explain in Appendix \ref{app_subsec_HS}, this process needs some additional choices of cochains. 
By the assumption that $E$ is rationally even, the resulting map on the differential cohomology level does not depend on the choices.  
\label{footnote_hoge}
}, the map $\widehat{\mathcal{G}}$ associates to every proper neat map of $p \colon N \to X$ of relative dimension $r$ with a {\it differential (tangential) $BG$-orientation} $g_p^\HS$ with a map
\begin{align}\label{eq_diff_push_HS}
    (p, g_p^\HS)_* \colon \widehat{E}^{*}_\HS (N; \iota_E) \to  \widehat{E}^{*-r}_\HS (X; \iota_E), 
\end{align}
called the {\it differential pushforward map}. 

\begin{rem}
As we explain in Appendix \ref{app_subsec_HS} and \ref{app_subsec_tangential}, the definition of (the tangential version of) differential $BG$-oriented maps in \cite{HopkinsSinger2005} differs from the differential stable relative $G$-structure in \cite[Definition 5.12]{YamashitaYonekura2021}. 
Fix a fundamental cocycle $\iota_{MTG} \in Z^0(MTG; V_{MTG}^\bullet)$. 
Given a proper smooth map $p \colon N \to X$, a topological tangential $BG$-orientation consists of a choice of embedding $N \hookrightarrow \R^k \times X$ for some $k$, a tubular neighborhood $W$ of $N$ in $\R^k \times X$ with a vector bundle structure $W \to N$, and a classifying map $\overline{W} := \Thom(W) \to MTG_{k-r}$. 
A differential tangential $BG$-orientation $g_p^\HS$ consists of its lift to a differential function
\begin{align}\label{eq_HS_orientation}
    t({g_p^\HS}) = (c, h, \omega) \colon \overline{W} \to (MTG_{k-r}, (\iota_{MTG})_{k-r}), 
\end{align}
Then the map \eqref{eq_diff_push_HS} is given by \eqref{eq_hat_mathcalG} and the Pontryagin-Thom construction. 
The resulting pushforward maps depend on the various choices. 
\end{rem}

\begin{rem}
However, using the assumption that $E$ is rationally even, in the case where $p$ is a submersion the situation is simple.  
First of all, the relative tangent bundle $T(p) = \ker(TN \to TX)$ makes sense, and we restrict our attention to the case where we are given a differential stable $G$-structure $g_p$ on $T(p)$ (as opposed to the more general notion of {\it differential stable relative tangential $G$-structure on $p$} in \cite[Definition 5.12]{YamashitaYonekura2021}). 
Then, associated to such $g_p$ there is a canonical set of choices of $g_q^\HS$ which gives the same map \eqref{eq_diff_push_HS}. 
We explain this point in details in Appendix \ref{app_diff_push}. 
We call such $g_p^\HS$ a {\it lift} of $g_p$ (Definition \ref{def_lift_HS_ori_tangential}). 
The map \eqref{eq_diff_push_HS} defined by any choice of a lift $g_p^\HS$ of $g_p$ is the unique map denoted by
\begin{align}\label{eq_diff_push_HS_preliminary}
    (p, g_p)_* := (p, g_p^\HS)_* \colon \widehat{E}^{*}_\HS (N; \iota_E) \to  \widehat{E}^{*-r}_\HS (X; \iota_E). 
\end{align}
We simply call it the {differential pushforward map} (Definition \ref{def_diff_push_tangential} and Proposition \ref{prop_push_HS=general_tangential}). 
\end{rem}

In the case where $p \colon N \to X$ is a submersion and equipped with a differential stable $G$-structure $g_p$ on $T(p)$, there is also the corresponding pushforward map on the level of differential forms. 
The Chern-Dold character \eqref{eq_chd_E} of the multiplicative genus $\mathcal{G} \in E^0(MTG)$ is the element
\begin{align}
    \mathrm{ch}(\mathcal{G}) \in H^0(MTG; V_E^\bullet). 
\end{align}
For example, for $\mathcal{G} = \tau \colon MT\mathrm{SO} \to H\Z$, the Chern-Dold chacacter is trivial, $1$. 
For $\mathcal{G} =\mathrm{ABS} \colon MT\mathrm{Spin}^c \to {K}$ and $\mathcal{G} =\mathrm{ABS} \colon MT\mathrm{Spin} \to {K}$, the Chern-Dold characters are the Todd polynomial and the $\widehat{A}$ polynomial, respectively. 
Applying the Chern-Weil construction in \eqref{eq_def_cw_hom_part2}, we get the Chern-Dold character form for the relative tangent bundle, 
\begin{align}
    \cw_{g_p}(\mathrm{ch}(\mathcal{G})) \in \Omega_{\mathrm{clo}}^0(N; \Ori(T(p)) \otimes_\R V_E^\bullet). 
\end{align}
Using this, the pushforward map on $\Omega^*(-; V_E^\bullet)$ is given by
\begin{align}\label{eq_push_form}
    \int_{N / X}-\wedge \cw_{g_p}(\mathrm{ch}(\mathcal{G}))  \colon \Omega^n(N; V_{E}^\bullet)  \to \Omega^{n-r}(X; V_{E}^\bullet). 
\end{align}
Restricted to the closed forms, the induced homomorphism on the cohomology, $H^n(N; V_{E}^\bullet) \to H^{n-r}(X; V_{E}^\bullet)$, is compatible with the Chern-Dold character \eqref{eq_chd_E} for $E$ and the topological pushforward \eqref{eq_top_push}. 
The differential pushforward map in \eqref{eq_diff_push_HS_preliminary} is compatible with the map \eqref{eq_push_form} (tangential version of \eqref{diag_diff_push_app}).

In particular, if $X = \pt$, for every $n$-dimensional differential stable tangential $G$-cycle $(M, g)$ over $\pt$, the differential pushforward map \eqref{eq_diff_push_HS_preliminary} is 
\begin{align}\label{eq_diff_push_pt_preliminary}
    (p_M, g)_*  \colon \widehat{E}^{*}_\HS (M; \iota_E) \to  \widehat{E}^{*-n}_\HS (\pt; \iota_E). 
\end{align}

As we explain in the last part of Appendix \ref{app_subsec_normal}, an important property of the pushforward is the following {\it Bordism formula}, relating the pushforward of differential forms \eqref{eq_push_form} on the bulk and the differential pushforward \eqref{eq_diff_push_pt_preliminary} on the boundary. 

\begin{fact}[{Bordism formula, \cite[Problem 4.235]{Bunke2013}}]\label{fact_bordism}
 For any morphism $[W, g_W] \colon (M_-, g_-) \to (M_+, g_+)$ in $\hBord_{n-1}(\pt)$, 
the following diagram commutes. 
\begin{align*}
    \xymatrix{
    \widehat{E}_\HS^*(W; \iota_E) \ar[r]^R \ar[d]^{(-i^*_{M_-}) \oplus i^*_{M_+}} &\Omega^*(W; V_E^\bullet)\ar[rr]^-{{\int_W -\wedge \cw_{g}(\mathrm{ch}(\mathcal{G}))}} & &\Omega^{*-n}(\pt; V_E^\bullet) \ar[d]^a\\
    \widehat{E}_\HS^*(M_-; \iota_E) \oplus \widehat{E}_\HS^*(M_+; \iota_E)\ar[rrr]^{(p_{M_-}, g_-)_* \oplus (p_{M_+}, g_+)_*} &&& \widehat{E}^{*-n+1}(\pt)
    }. 
\end{align*}

\end{fact}

\begin{ex}
In the case $\mathcal{G} = \tau \colon MT\mathrm{SO} \to H\Z$, the nontrivial degree of pushforwards $(p_M, g)_*\colon \widehat{H}^{\dim M+1}(M; \Z) \to \widehat{H}^1(\pt; \Z) \simeq \R/\Z$ are called the {\it higher holonomy function} which appears in the definition of Chern-Simons invariants. 
In terms of the Cheeger-Simons model of $\widehat{H\Z}$ in terms of differential characters \cite{CheegerSimonsDiffChar}, it is given by the evaluation on the fundamental cycle. 
In particular for the case $\dim M = 1$ it is the usual holonomy, and the Bordism formula is satisfied because of the relation between the curvature and the holonomy for $U(1)$-connections. 
\end{ex}

\begin{ex}
In the case $\mathcal{G}=\mathrm{ABS} \colon MT\mathrm{Spin}^c \to {K}$,
Freed and Lott \cite{FL2010} constructed a model of $\widehat{K}$ in terms of hermitian vector bundles with hermitian connections, and the refinement of pushforwards when $\dim M$ is odd, $(p_M, g)_* \colon \widehat{K}^0(M) \to \widehat{K}^{-\dim M}(\pt) \simeq \R/\Z$, is given by the reduced eta invariants. 
The Bordism formula is a consequence of the Atiyah-Patodi-Singer index theorem. 
\end{ex}

\subsection{Differential Pushforwards in terms of functors}\label{subsec_genus_functor}

As a preparation to the main Subsection \ref{subsec_genus_main}, in this subsection we translate the data of differential pushforwards into functors from $\hBord_-(-)$. 

\begin{defn}\label{def_genus_functor}
In the above settings, let $X$ be a manifold, $k$ be an integer and $\widehat{e} \in \widehat{E}_\HS^k(X; \iota_E)$.  
Let $n$ be an integer with $k + n - 1 \ge 0$. 
Then define the functor of Picard groupoids,
\begin{align}\label{eq_genus_functor}
    T_{{\mathcal{G}}, \widehat{e}} \colon \hBord_{k+n-1}(X) \to \left(
    V_E^{-n} \xrightarrow{a} \widehat{E}_\HS^{1-n}(\pt; \iota_E)
    \right),
\end{align}
by the following. 
\begin{itemize}
    \item On objects, we set 
    \begin{align}\label{eq_def_T_Ge}
        T_{{\mathcal{G}}, \widehat{e}}(M, g, f) := (p_{M}, g)_* f^*(\widehat{e}) \in \widehat{E}_\HS^{1-n}(\pt; \iota_E)
    \end{align}
    \item On morphisms, we set 
    \begin{align*}
        T_{{\mathcal{G}}, \widehat{e}}([W, g_W, f_W]):=  \cw( R(\widehat{e}) \wedge \mathrm{ch}(\mathcal{G}))([W, g_W, f_W]). 
    \end{align*}
    Here $R(\widehat{e}) \in \Omega_{\mathrm{clo}}^k(X; V_E^\bullet)$ is the curvature of $\widehat{e}$ and we use \eqref{eq_def_cw_omega_mor_part2}. 
\end{itemize}
The well-definedness of the functor follows by the Bordism formula in Fact \ref{fact_bordism}. 
\end{defn}

As is easily shown by the Bordism formula, the formula \eqref{eq_def_T_Ge} defines the homomorphism 
\begin{align}\label{eq_T_Ge_hom}
    T_{{\mathcal{G}}, \widehat{e}} \colon \mathcal{C}^{G_\nabla}_{k+n-1}(X) \to \widehat{E}_\HS^{1-n}(\pt; \iota_E). 
\end{align}

As expected, the transformation \eqref{eq_genus_functor} is induced by the first arrow in \eqref{eq_genus_theory_2}. 
To show this, 
first remark that for any spectrum $F$ and its any fundamental cycle $\iota_F$, the forgetful functor gives the equivalence of Picard groupoids, 
\begin{align}\label{eq_proof_genus_functor_1}
    \pi_{\le 1}((F; \iota_F)^\pt_n) \simeq \pi_{\le 1}(F_n), 
\end{align}
where the left hand side means the simplicial fundamental groupoid, whose objects are differential functions $t_\pt \colon \pt \to (F; \iota_F)_n$, and morphisms are bordism classes of differential functions $t_{I} \colon I \to (F; \iota_F)_n$. 
The right hand side is the fundamental groupoid for the space $F_n$, which is equipped with the structure of a Picard groupoid by \cite[Example B.7]{HopkinsSinger2005}. 

We have a functor of Picard groupoids\footnote{The right hand side is the Picard groupoid associated to the homomorphism $a \colon V_E^{-n} \to \widehat{E}_\HS^{1-n}(\pt; \iota_E)$ of abelian groups. 
In general, a homomorphism $\del \colon A \to B$ between abelian groups associates a Picard groupoid $(A \xrightarrow{\del} B)$, whose objects are elements of $B$, and a morphism from $b$ to $b'$ is given by an element $a \in A$ such that $b' - b = \del(a)$. },
\begin{align}\label{eq_genus_functor_E_pt_hat}
   \mathrm{ev} \colon \pi_{\le 1}\left((E; \iota_E)^\pt_{1-n} \right) \to \left(
    V_E^{-n} \xrightarrow{a} \widehat{E}_\HS^{1-n}(\pt; \iota_E)
    \right) ,
\end{align}
given by assigning the element $[t_\pt] \in \widehat{E}_\HS^{1-n}(\pt; \iota_E)$ for an object and the integration of the curvature $R([t_{I}]) \in \Omega_{\mathrm{clo}}^{1-n}(I; V_E^\bullet)$ for a morphism.

\begin{prop}\label{prop_genus_functor}
The functor \eqref{eq_genus_functor} of Picard groupoids is naturally isomorphic to the following composition, 
\begin{align}\label{eq_prop_genus_functor}
    \hBord_{k+n-1}(X) \simeq \pi_{\le 1}(L( X^+\wedge MTG)_{1-k-n})
    \xrightarrow{e \wedge \mathcal{G}} \pi_{\le 1}(E_{1-n}) \to  \left(
    V_E^{-n} \xrightarrow{a} \widehat{E}_\HS^{1-n}(\pt; \iota_E)
    \right), 
\end{align}
where the first arrow is the equivalence in \eqref{eq_lem_cat_equivalence_part2}, and the last arrow is the composition of \eqref{eq_proof_genus_functor_1} and \eqref{eq_genus_functor_E_pt_hat}. 
\end{prop}
\begin{proof}

Choose a differential function $t(\widehat{e})\colon X \to (E_k; (\iota_E)_k)$ representing $\widehat{e}$. 
For each object $(M, g, f)$ in $\hBord_{k+n-1}(X)$, choose a Hopkins-Singer's differential $G$-structure $g^\HS$ lifting $g$. 
By the discussion in Appendix \ref{app_subsec_HS} and its tangential variant in Appendix \ref{app_subsec_tangential}, we get a functor
\begin{align}\label{eq_proof_genus_functor_2}
    \hBord_{k+n-1}(X) \to \pi_{\le 1}((E_k)^+ \wedge MTG_{1-k-n}; (\iota_E)_k \cup V_{\mathcal{G}}(\iota_{MTG})_{1-k-n})^\pt). 
\end{align}
Indeed, for objects, given $(M, g, f)$ with the chosen lift $g^\HS$, denote the underlying embedding and tubular neighborhood by $M \subset U \subset \R^N$. 
We have differential functions $f^* t(\widehat{e}) \colon M \to (E; \iota_E)_k$ and $t({g^\HS}) \colon \overline{U} \to (MTG; \iota_{MTG})_{N-(k+n-1)}$. 
Applying the ($MTG$-version of the) left vertical arrow of \eqref{diag_HS_multi} to them and using the open embedding $U \hookrightarrow \R^N$ (the Pontryagin-Thom collapse), we get the differential function $\pt \to (E_k \wedge MTG_{1-k-n}; (\iota_E)_k \cup V_{\mathcal{G}}(\iota_{MTG})_{1-k-n})$. 

For morphisms $[W, g_W, f_W] \colon (M_-, g_-, f_-) \to (M_+, g_+, f_+) $, choose any representative $(W, g_W, f_W)$ and smooth map $p_W \colon W \to I$ (not necessarily a submersion) so that it coincides with a collar coordinates of each objects $(M_\pm, g_\pm, f_\pm)$ near the endpoints, respectively. 
The structure $g_W$ induces $g_{p_W}$, in particular the topological structure $g_{p_W}^{\mathrm{top}}$, on $p_W$. 
Take any Hopkins-Singer's differential tangential $BG$-oorientation $g_{p_W}^{\HS}$ (Appendix \ref{app_subsec_tangential}) for $p_W$ which coincides with the chosen lifts at the boundary, and whose underlying map classifies $g_{p_W}^{\mathrm{top}}$. 
Then applying the same procedure as that we did for objects above, we get a differential function $I \to (E_k \wedge MTG_{1-k-n}; (\iota_E)_k \cup V_{\mathcal{G}}(\iota_{MTG})_{1-k-n})$ which restricts at the boundary to the ones assigned to objects above. 
Since any of the choices we have made is unique up to bordisms, the resulting morphism in the right hand side of \eqref{eq_proof_genus_functor_2} is uniquely determined. 
This gives the desired functor. 


By Definition \ref{def_genus_functor} and Proposition \ref{prop_push_HS=general_tangential}, the functor $T_{{\mathcal{G}}, \widehat{e}}$ coincides with the composition of \eqref{eq_proof_genus_functor_2} with
\begin{align}
     \pi_{\le 1}\left( \left( E_k \wedge MTG_{1-k-n} ;  (\iota_E)_k \cup V_{\mathcal{G}}(\iota_{MTG})_{1-k-n} \right)^\pt\right) \xrightarrow{\widehat{\mathcal{G}} } \pi_{\le 1}((E; \iota_E)^\pt_{1-n}) \\
     \xrightarrow{\mathrm{ev}} \left(
    V_E^{-n} \xrightarrow{a} \widehat{E}_\HS^{1-n}(\pt; \iota_E)
    \right).  \notag
\end{align}
The fact that it is naturally isomorphic to \eqref{eq_prop_genus_functor} is just the cosequence of the fact that $\widehat{e}$ and $\widehat{\mathcal{G}}$ are refinements of $e$ and $\mathcal{G}$, respectively. 
This completes the proof. 
\end{proof}

\subsection{The construction and the proof}\label{subsec_genus_main}
In this main subsection, we state and prove the main result of this article. 
Let $G$ be a multiplicative tangential structure groups, $E$ be a rationally even ring spectrum and $\mathcal{G} \colon MTG \to E$ be a homomorphism of ring spectra. 
Fix an integer $n$ such that $E^{1-n}(\pt) \otimes \R = 0$.  
We construct a natural transformation
\begin{align}\label{eq_goal_theory_2}
    \Phi_{\mathcal{G}} \colon \widehat{E}_{\mathrm{HS}}^*(-; \iota_E) \otimes IE^n(\pt) \to (\widehat{I\Omega^G_{\mathrm{dR}}})^{*+n}(-), 
\end{align}
on $\mathrm{Mfd}^{\mathrm{op}}$ (Definition \ref{def_theory}).

The main result of this paper is the following topological characterization of the transformation \eqref{eq_goal_theory_2}. 
\begin{thm}[{=Theorem \ref{thm_genus_theory}}]\label{thm_genus_theory_2}
In the above settings, 
let $X$ be a manifold and $k$ be an integer. 
For $\widehat{e} \in \widehat{E}_\HS^k(X; \iota_E)$ and $\beta \in IE^n(\pt)$, the element $I(\Phi_{\mathcal{G}}(\widehat{e} \otimes \beta)) \in (I\Omega^G)^{k+n}(X) = [X^+\wedge MTG ,  \Sigma^{k+n}I\Z]$ coincides with the following composition, 
\begin{align}\label{eq_genus_theory_2}
    X^+\wedge MTG \xrightarrow{e \wedge \mathcal{G}} 
    \Sigma^k E \wedge E \xrightarrow{\mathrm{multi}}
    \Sigma^k E \xrightarrow{\beta} \Sigma^{k+n}I\Z. 
\end{align}
Here we denoted $e := I(\widehat{e}) \in E^k(X)$. 
\end{thm}

As a preparation, we show that there exists a canonical homomorphism\footnote{
The existsnce of a canonical pairing $IE^{-n}(\pt) \otimes \widehat{E}^{1-n}_\HS(\pt; \iota_E) \to \R/\Z$ is used in \cite[Proposition 6]{FMS07}, in particular in the last arrow of the second line of the proof of that proposition. 
They do not state any condition on $E$, but they use the assumption $V_E^{1-n}=0$ implicitely. 
}
\begin{align}\label{eq_Moore}
    s \colon IE^n(\pt) \to \mathrm{Hom}_{\mathrm{Ch}}\left(
    \left(
    V_E^{-n} \xrightarrow{a} \widehat{E}_\HS^{1-n}(\pt; \iota_E)
    \right) , (\R \to \R/\Z) 
    \right) ,
\end{align}
where $\Hom_{\mathrm{Ch}}$ denotes the group of chain maps of complexes of abelian groups. 
Indeed, by \cite[(4.57)]{HopkinsSinger2005}, we have a canonical isomorphism\footnote{
This isomorphism does not follow from the axiom of differential cohomology theory in \cite{BunkeSchick2010}. 
For more on this point, see \cite[Section 5]{BunkeSchick2010}. 
}
\begin{align}\label{eq_HS_flat_general}
    \ker\left(R \colon \widehat{E}^{*}_\HS(-; \iota_E) \to \Omega^{*}_{\mathrm{clo}}(-; V_E^\bullet) \right) \simeq E^{*-1}(-; \R/\Z). 
\end{align}
Here, for any abelian group $\mathbb{G}$, the cohomology theory $E^*(-; \mathbb{G})$ is represented by the spectrum $E\mathbb{G} := E \wedge S\mathbb{G}$, where $S\mathbb{G}$ is the Moore spectrum.
As explained there, this is because the differential function complexes can be fits into the homotopy Cartesian square \cite[(4.12)]{HopkinsSinger2005}. 
Applied to $\pt$ and $*=1-n$, we get the identification
\begin{align}\label{eq_HS_flat}
    \widehat{E}^{1-n}_\HS(\pt; \iota_E) = \ker\left(R \colon \widehat{E}^{1-n}_\HS(\pt; \iota_E) \to V_E^{1-n} \right) \simeq E^{-n}(\pt; \R/\Z). 
\end{align}
An element $\beta \in IE^n(\pt) = [E, \Sigma^n I\Z]$ induces the element $\beta_{\mathbb{G}} \in [E\mathbb{G}, \Sigma^n I\Z \wedge S\mathbb{G}]$ for any $\mathbb{G}$, and using $I\Z \wedge S\R \simeq  H\R$ and $I\Z \wedge S\R/\Z \simeq I\R/\Z$, we get the induced homomorphisms on $\pt$, which we also denote as
\begin{align}
    \beta_{\R} &\colon V_E^{-n} = E^{-n}(\pt; \R) \to \R, \label{eq_beta_R} \\
    \beta_{\R/\Z} &\colon \widehat{E}^{1-n}_\HS(\pt; \iota_E) \xrightarrow[\eqref{eq_HS_flat}]{\simeq} E^{-n}(\pt; \R/\Z) \to \R/\Z,  \label{eq_beta_R/Z}
\end{align}
The homomorphism \eqref{eq_beta_R} coincides with the one obtained by the map $IE^n(X) \to \mathrm{Hom}(E_{n}(X), \R) $ in \eqref{eq_exact_IE_part2}. 
The homomorphism \eqref{eq_Moore} is given by mapping $\beta$ to the pair $(\beta_\R, \beta_{\R/\Z})$. 
The well-definedness follows by the construction.

On the other hand, by \cite[Corollary B.17]{HopkinsSinger2005}\footnote{
Note that there is an obvious typo of the degree in the statement of \cite[Corollary B.17]{HopkinsSinger2005}. 
} (also see \cite[Fact 2.6]{YamashitaYonekura2021} and the explanations there), we have an isomorphism for any spectra $E$, 
\begin{align}\label{eq_model_IZ_part2}
    IE^n(\pt) \simeq \pi_0\mathrm{Fun}_{\mathrm{Pic}}\left(\pi_{\le 1}(E_{1-n}), (\R \to \R/\Z) 
    \right), 
\end{align}
where $\pi_0 \mathrm{Fun}_{\mathrm{Pic}}$ means the group of natural isomorphism classes of functors of Picard groupoids. 
By \eqref{eq_proof_genus_functor_1}, \eqref{eq_genus_functor_E_pt_hat} and \eqref{eq_model_IZ_part2}, we get a homomorphism
\begin{align}\label{eq_ev_hom}
     \mathrm{ev}_* \colon \pi_0\mathrm{Fun}_{\mathrm{Pic}}\left( \left(
    V_E^{-n} \xrightarrow{a} \widehat{E}_\HS^{1-n}(\pt; \iota_E)
    \right), (\R \to \R/\Z) 
    \right) \to IE^n(\pt). 
\end{align}
It directly follows from the definition of the identification \eqref{eq_HS_flat_general} that we have
\begin{align}\label{eq_ev_s_id}
    \mathrm{id} = \mathrm{ev}_* \circ s \colon IE^n(\pt) \to IE^{n}(\pt). 
\end{align}

\begin{defn}[{$\Phi_{\mathcal{G}}$}]\label{def_theory}
In the settings explained in the beginning of this subsection, for each manifold $X$
we define a homomorphism of abelian groups
\begin{align}
    \Phi_{\mathcal{G}} \colon \widehat{E}_\HS^*(X; \iota_E) \otimes IE^n(\pt) \to (\widehat{I\Omega^G_{\mathrm{dR}}})^{*+n}(X), 
\end{align}
by the following. 
For $\widehat{e}\in  \widehat{E}_\HS^k(X; \iota_E)$ and $\beta \in IE^n(\pt)$, set
$\Phi_{\mathcal{G}}(\widehat{e}\otimes \beta) := (\beta_\R(R(\widehat{e}) \wedge \mathrm{ch}(\mathcal{G})), \beta_{\R/\Z} \circ T_{{\mathcal{G}}, \widehat{e}}) \in (\widehat{I\Omega^G_{\mathrm{dR}}})^{k+n}(X)$, where
\begin{itemize}
    \item The element $\beta_\R (R(\widehat{e}) \wedge \mathrm{ch}(\mathcal{G})) \in \Omega_{\mathrm{clo}}^{n+k}(X; N_{G}^\bullet)$ is obtained by applying \eqref{eq_beta_R} on the coefficient of $R(e) \wedge \mathrm{ch}(\mathcal{G}) \in \Omega_{\mathrm{clo}}^k(X; H^*(MTG; V_E^\bullet))$. 
    \item $\beta_{\R/\Z} \circ T_{{\mathcal{G}}, \widehat{e}}$ is the composition of \eqref{eq_T_Ge_hom} and \eqref{eq_beta_R/Z}.  
\end{itemize}
The fact that the pair $ (\beta_\R(R(\widehat{e}) \wedge \mathrm{ch}(\mathcal{G})), \beta_{\R/\Z}(T_{{\mathcal{G}}, \widehat{e}} ))$ satisfies the compatibility condition follows from the well-definedness of \eqref{eq_Moore} and the fact that $T_{\mathcal{G}, e}$ in Definition \ref{def_genus_functor} is a functor. 
\end{defn}

Now we prove Theorem \ref{thm_genus_theory_2}. 
\begin{proof}[Proof of Theorem \ref{thm_genus_theory_2}(=Theorem \ref{thm_genus_theory})]
We use the argument in \cite[Subsection 4.2]{YamashitaYonekura2021}.  
Recall that, for an element $(\omega, h) \in (\widehat{I\Omega^G_\dR})^N(X)$ we associated a functor $F_{(\omega, h)} \colon \hBord_{N-1}(X) \to (\R \to \R/\Z)$ in \cite[(4.46)]{YamashitaYonekura2021}. 
In the proof of \cite[Theorem 4.51]{YamashitaYonekura2021}, 
the natural isomorphism
\begin{align}\label{eq_isom_IOmega_dR_part2}
   I\Omega^G \simeq I\Omega^G_\dR,
\end{align}
where for the former we use the model of $I\Z$ by \cite[Corollary B.17]{HopkinsSinger2005}, was given as follows. 
Using the equivalence \eqref{eq_lem_cat_equivalence_part2}, we have $(I\Omega^G)^{N}(X) = \pi_0\mathrm{Fun}_{\mathrm{Pic}}\left(\pi_{\le 1}(\hBord_{N-1}(X) \to (\R \to \R/\Z)\right)$. 
The map \eqref{eq_isom_IOmega_dR_part2} is given by mapping the isomorphism class of the functor $F_{(\omega, h)}$ to $I(\omega, h) \in ({I\Omega^G_\dR})^N(X)$. 

Now fix $\widehat{e} \in \widehat{E}_\HS^k(X; \iota_E)$ and $\beta \in IE^n(\pt)$. 
By Definitions \ref{def_theory} and \ref{def_genus_functor}, the functor associated to $\Phi_{\mathcal{G}}(\widehat{e}\otimes \beta)$ coincides with the following composition. 
\begin{align}\label{eq_proof_genus_theory}
    F_{\Phi_{\mathcal{G}}(\widehat{e}\otimes \beta)} \colon \hBord_{k+n-1}(X) \xrightarrow{T_{{\mathcal{G}}, \widehat{e}}} \left(
    V_E^{-n} \xrightarrow{a} \widehat{E}_\HS^{1-n}(\pt; \iota_E)
    \right) \xrightarrow{s(\beta) = (\beta_\R, \beta_{\R/\Z})} (\R \to \R/\Z). 
\end{align}

Combining this with Proposition \ref{prop_genus_functor} and \eqref{eq_ev_s_id}, we see that, under the equivalence $\hBord_{k+n-1}(X) \simeq \pi_{\le 1}(L(X^+\wedge MTG)_{1-k - n})$, \eqref{eq_proof_genus_theory} coincides with
\begin{align*}
    \pi_{\le 1}(L(X^+\wedge MTG)_{1-k - n}) \xrightarrow{e \wedge \mathcal{G}} \pi_{\le 1}(E_{1-n}) \xrightarrow{\beta} (\R \to \R/\Z), 
\end{align*}
up to a natural isomorphism. 
This completes the proof of Theorem \ref{thm_genus_theory_2}. 
\end{proof}

\subsection{Examples}\label{subsec_genus_example}
\subsubsection{The holonomy theory (1) : \cite[Example 4.54]{YamashitaYonekura2021}}\label{subsec_hol_target}
Here we explain the easiest example of the ``Holonomy theory (1)'' which appeared in \cite[Example 4.54]{YamashitaYonekura2021}. 
This corresponds to the case $E = H\Z$, $\mathcal{G} = \tau \colon MT\mathrm{SO} \to H\Z$ is the usual orientation, and $n = 0$. 

Recall that, given a manifold $X$ and a hermitian line bundle with unitary connection $(L, \nabla)$ over $X$, we get the element
\begin{align}\label{Hol_class}
    (c_1(\nabla), \mathrm{Hol}_{\nabla}) \in (\widehat{I\Omega^{\mathrm{SO}}_{\mathrm{dR}}})^{2}(X). 
\end{align}
On the other hand, in the case $E = H\Z$ we have the canonical choice of an element in $IH\Z^0(\pt)$, namely the Anderson self-duality element $\gamma_H \in IH\Z^0(\pt)$. 
Thus we have the homomorphism
\begin{align*}
    \Phi_{\tau}(-\otimes \gamma_H) \colon \widehat{H}^2(X; \Z) \to (\widehat{I\Omega^{\mathrm{SO}}_{\mathrm{dR}}})^{2}(X). 
\end{align*}
Using the model of $\widehat{H\Z}^2$ in terms of hermitian vector bundles with $U(1)$-connections (for example see \cite[Example 2.7]{HopkinsSinger2005}), the pair $(L, \nabla)$ defines a class $[L, \nabla] \in \widehat{H}^2(X; \Z)$. 
We have the following. 
\begin{prop}\label{prop_hol}
We have the following equality in $(\widehat{I\Omega^{\mathrm{SO}}_{\mathrm{dR}}})^{2}(X)$, 
\begin{align}\label{eq_prop_hol}
    (c_1(\nabla), \mathrm{Hol}_{\nabla}) = \Phi_{\tau}([L, \nabla]\otimes \gamma_H). 
\end{align}
Moreover, the element $I(c_1(\nabla), \mathrm{Hol}_{\nabla}) \in ({I\Omega^{\mathrm{SO}}_{\mathrm{dR}}})^{2}(X)$ coincides with the following composition, 
\begin{align*}
    X^+ \wedge MT\mathrm{SO} \xrightarrow{ c_1(L) \wedge \tau} 
    \Sigma^2 H\Z\wedge H\Z  \xrightarrow{\mathrm{multi}}
    \Sigma^2 H\Z \xrightarrow{\gamma_H} \Sigma^{2}I\Z. 
\end{align*}
\end{prop}
\begin{proof}
The last statement follows from \eqref{eq_prop_hol} and Theorem \ref{thm_genus_theory_2}. 
The equality \eqref{eq_prop_hol} follows from 
the fact that the self-duality element $\gamma_H$ induces the canonical isomorphism $\widehat{H}^1(\pt; \Z) \simeq \R/\Z$ and $H^0(\pt; \Z) \simeq \R$,
together with the following well-known facts about $\widehat{H\Z}$ (for example see \cite[Section 2.4]{HopkinsSinger2005}). 
The element $[L, \nabla] \in \widehat{H}^2(X; \Z)$ satisfies
\begin{align*}
    \gamma_H \circ R([L, \nabla]) = c_1(\nabla)\in \Omega^2(X), 
    \end{align*}
and, given a map $f \colon M \to X$ from an oriented $1$-dimensional closed manifold $(M, g)$, the pushforward $(p_M, g)_* \colon \widehat{H}^2(M; \Z) \to \widehat{H}^1(\pt; \Z) \xrightarrow[\simeq]{\gamma_H} \R/\Z$
    \begin{align*}
    \gamma_H \circ (p_M, g)_* f^*[L, \nabla] = \mathrm{Hol}_{f^*\nabla}. 
\end{align*}
\end{proof}

\subsubsection{The classical Chern-Simons theory : \cite[Example 4.56]{YamashitaYonekura2021}}\label{subsec_CCS}
Here we explain the classical Chern-Simons theory which appeared in \cite[Example 4.56]{YamashitaYonekura2021}. 
This is essentially a generalization of Subsection \ref{subsec_hol_target}, 
corresponding to the case $E = H\Z$, $\mathcal{G} = \tau \colon MT\mathrm{SO} \to H\Z$ is the usual orientation, and $n = 0$. 

Recall that, given a compact Lie group $H$ and an element $\lambda \in H^n(BH; \Z)$, the corresponding classical Chern-Simons theory of level $\lambda$ is defined by choosing an $(n+1)$-classifying object $(\mathcal{E}, \mathcal{B}, \nabla_{\mathcal{E}})$ in the category of manifolds with principal $H$-bundles with connections, and fixing an element $\widehat{\lambda} \in \widehat{H}^n(\mathcal{B}; \Z)$ lifting $\lambda$. 
Then we have the element
\begin{align}\label{eq_CCS_elem}
    (1 \otimes\lambda_\R, h_{\mathrm{CS}_{\widehat{\lambda}}}) \in (\widehat{I\Omega^{\mathrm{SO}\times H}_{\mathrm{dR}}})^{n}(\pt), 
\end{align}
whose equivalence class in $(I\Omega^{\mathrm{SO} \times H}_\dR)^n(\pt)$ does not depend on the lift $\widehat{\lambda}$.

\begin{prop}\label{prop_CCS}
The element 
$I(1 \otimes\lambda_\R, h_{\mathrm{CS}_{\widehat{\lambda}}} )
    \in (I\Omega^{\mathrm{SO}\times H}_{\mathrm{dR}})^{n}(\pt)$
coincides with the following composition, 
\begin{align*}
     BH^+\wedge MT\mathrm{SO} \xrightarrow{\lambda \wedge \tau} 
    \Sigma^n H\Z\wedge  H\Z  \xrightarrow{\mathrm{multi}}
    \Sigma^n H\Z \xrightarrow{\gamma} \Sigma^{n}I\Z. 
\end{align*}
\end{prop}

\begin{proof}
The classifying map induces an equivalence $\pi_{\le 1}(L(\mathcal{B}^+ \wedge MT\mathrm{SO})_{n-1}) \simeq \pi_{\le 1}(L(BH^+ \wedge MT\mathrm{SO})_{n-1})$. 
Moreover, by the pullback of the universal connection $\nabla_{\mathcal{E}}$ it is refined to a functor of Picard groupoids, 
\begin{align}\label{eq_proof_CCS_0}
    h\mathrm{Bord}^{\mathrm{SO}}_{n-1}(\mathcal{B}) \xrightarrow{\simeq} h\mathrm{Bord}^{\mathrm{SO} \times H}_{n-1}(\pt)
\end{align}
which is naturally isomorphic to the above one under the equivalences $h\mathrm{Bord}^{\mathrm{SO}}_{n-1}(X) \simeq \pi_{\le 1}(L(X^+ \wedge MT\mathrm{SO})_{n-1})$. 

We have the element 
\begin{align}\label{eq_proof_CCS_1}
    \Phi_{\tau}(\widehat{\lambda}\otimes \gamma_H) \in (\widehat{I\Omega^{\mathrm{SO}}_{\mathrm{dR}}})^{n}(\mathcal{B}). 
\end{align}
Recall that an element $(\omega, h)\in (\widehat{I\Omega^{G}_{\mathrm{dR}}})^{n}(X)$ associates a functor $F_{(\omega, h)} \colon \hBord_{n-1}(X) \to (\R \to \R/\Z)$ by \cite[(4.46)]{YamashitaYonekura2021}. 
We claim that the functors associated to the elements \eqref{eq_proof_CCS_1} and \eqref{eq_CCS_elem} are related by
\begin{align*}
    F_{\Phi_{\tau}(\widehat{\lambda}\otimes \gamma_H)} \colon h\mathrm{Bord}^{\mathrm{SO}}_{n-1}(\mathcal{B}) \xrightarrow{\eqref{eq_proof_CCS_0}} h\mathrm{Bord}^{\mathrm{SO} \times H}_{n-1}(\pt) \xrightarrow{F_{ (1 \otimes\lambda_\R, h_{\mathrm{CS}_{\widehat{\lambda}}})}} (\R\to \R/\Z). 
\end{align*}
Indeed, this follows from the fact that the Chern-Simons invariants are given by the pushforward in differential ordinary cohomology \cite[(4.58)]{YamashitaYonekura2021}. 
Applying Theorem \ref{thm_genus_theory_2} to the element \eqref{eq_proof_CCS_1}, we get the result. 
\end{proof}

\subsubsection{The theory of massive free complex fermions : \cite[Example 4.62]{YamashitaYonekura2021}}\label{subsec_complex_eta}
Here we explain the example of the theory on massive free complex fermions which appeared in \cite[Example 4.62]{YamashitaYonekura2021}. 
This example corresponds to the case $E = K$, $\mathcal{G} = \mathrm{ABS} \colon MT\mathrm{Spin}^c \to K$ and $n = 0$. 

Recall that, given a hermitian vector bundle with unitary connection $(W, \nabla^W)$ over a manifold $X$, we get an element
\begin{align*}
    \left((\mathrm{Ch}(\nabla^W) \otimes \mathrm{Todd})|_{2k}, \overline{\eta}_{\nabla^W}\right) \in \left(\widehat{I\Omega^{\mathrm{Spin}^c}_{\mathrm{ph}}}\right)^{2k}(X) \simeq  \left(\widehat{I\Omega^{\mathrm{Spin}^c}_{\mathrm{dR}}}\right)^{2k}(X). 
\end{align*}
On the other hand, in the case $E = K$ we have the canonical choice of an element in $IK^0(\pt)$, namely the self-duality element $\gamma_K \in IK^0(\pt)$. 
Thus we have the homomorphism
\begin{align}
    \Phi_{\mathrm{ABS}}(-\otimes \gamma_K) \colon \widehat{K}^{2k}(X) \to (\widehat{I\Omega^{\mathrm{Spin}^c}_{\mathrm{dR}}})^{2k}(X). 
\end{align}
Using the model of $\widehat{K}$ in terms of hermitian vector bundles with unitary connections by Freed-Lott (\cite{FL2010}), we have the class $[W, h^W, \nabla^W, 0] \in \widehat{K}^0(X) \simeq \widehat{K}^{2k}(X)$. 

\begin{prop}\label{prop_eta}
We have the following equality in $(\widehat{I\Omega^{\mathrm{Spin}^c}_{\mathrm{dR}}})^{2k}(X)$, 
\begin{align}\label{eq_prop_eta}
    ((\mathrm{Ch}(\nabla^W)\otimes \mathrm{Todd})|_{2k}, \bar{\eta}_{\nabla^W})= \Phi_{\mathrm{ABS}}([W, h^E, \nabla^E, 0]\otimes \gamma_K). 
\end{align}
Moreover, the element $I ((\mathrm{Ch}(\nabla^W)\otimes \mathrm{Todd})|_{2k}, \bar{\eta}_{\nabla^W}) \in ({I\Omega^{\mathrm{Spin}^c}})^{2k}(X)$ coincides with the following composition, 
\begin{align*}
    X^+ \wedge MT\mathrm{Spin}^c \xrightarrow{[E]\wedge \mathrm{ABS}} 
    K\wedge K \xrightarrow{\mathrm{multi}}
    K \xrightarrow[\simeq]{\mathrm{Bott}} \Sigma^{2k}K \xrightarrow{\gamma_K} \Sigma^{2k}I\Z. 
\end{align*}
\end{prop}
\begin{proof}
The last statement follows from \eqref{eq_prop_eta} and Theorem \ref{thm_genus_theory_2}. 
Denote the Bott element by $u \in K^{-2}(\pt)$. 
The equality \eqref{eq_prop_eta} follows from 
the fact that the self-duality element $\gamma_K$ induces the canonical isomorphism
$\widehat{K}^1(\pt) \simeq \R/\Z$ and $K^0(\pt) \simeq \Z$, together with the following facts about $\widehat{K}$ in \cite{FL2010}. 
The element $ [W, h^W, \nabla^W,  0] \in \widehat{K}^{0}(X)$ satisfies
\begin{align*}
    R([W, h^E, \nabla^E, 0]) = \mathrm{Ch}(\nabla^W)\in \Omega^0(X; V_{K}^\bullet) = \Omega^0(X; \R[u, u^{-1}]), 
    \end{align*}
and, given a map $f \colon M \to X$ from an oriented $(2k-1)$-dimensional closed manifold with a physical Spin$^c$-structure $(M, g)$, the pushforward $(p_M, g)_* \colon \widehat{K}^{0}(M) \to \widehat{K}^{-2k+1}(\pt)$ is given by
    \begin{align*}
    (p_M, g)_* f^*[W, h^W, \nabla^W, 0] = \bar{\eta}_{\nabla^W}(M, g, f) \cdot u^{k} \in \widehat{K}^{-2k+1}(\pt) = (\R/\Z) \cdot u^k. 
\end{align*}
\end{proof}

\subsubsection{An interpretation of Subsection \ref{subsec_complex_eta} - Taking anomaly theories of free spinor field theories}\label{subsec_anomaly}

Here we explain an interpretation of the result in Subsection \ref{subsec_complex_eta} in terms of {\it anomalies} of free spinor field theories. 

We briefly recall the explanation in \cite[Lecture 11]{Freed19} and \cite[Section 9]{Freed:2016rqq} about free spinor field theory and its anomalies. 
A real spinor representation $\mathbb{S}$ of $\mathrm{Spin}_{1, d-1}$ and a (contractible) choice of nonnegative symmetric invariant bilinear pairing $\Gamma \colon \mathbb{S} \times \mathbb{S} \to \R^{1, d-1}$ determine an $d$-dimensional possibly {\it anomalous} theory called the free real spinor field theory $F_{(\mathbb{S}, \Gamma)}$. 
The spinor representation $\bS$ gives an element $[\bS] \in \pi_{2-d}{KO}$. 

An {\it anomalous} $d$-dimensional field theory is formulated as a boundary theory of $(d+1)$-dimensional {\it invertible} field theory called the associated {\it anomaly theory}, which is classified by $(I\Omega^G)^{d+2}$. 
In this case the relevant structure group is $G=\mathrm{Spin}$. 
As explained in the references, the anomaly theory associated to $F_{(\bS, \Gamma)}$ has partition function given by suitable fraction of exponentiated reduced eta invariants, depending on $d$ mod $8$. 

Freed and Hopkins suggested the following conjecture. 
\begin{conj}[{\cite[Conjecture 11.23]{Freed19}, \cite[Conjecture 9.70]{Freed:2016rqq}}]\label{conj_anomaly_fermion}
    The $(d+1)$-dimensional anomaly theory associated to the $d$-dimensional free real spinor field theory $F_{(\mathbb{S}, \Gamma)}$ correponds to the element in $(I\Omega^{\mathrm{Spin}})^{d+2}(\pt)$ given by the following composotion. 
    \begin{align}\label{eq_Phi_KO}
          MT\mathrm{Spin} \xrightarrow{[\bS] \wedge \mathrm{ABS}} 
    \Sigma^{d-2} KO \wedge KO\xrightarrow{\mathrm{multi}}
    \Sigma^{d-2} KO \xrightarrow{\gamma_{KO}} \Sigma^{d+2}I\Z. 
    \end{align}
    Here $\mathrm{ABS} \colon MT\mathrm{Spin} \to KO$ is the Atiyah-Bott-Shapiro map and $\gamma_{KO} \in IKO^{4}(\pt)$ is the Anderson self-duality element for the $KO$-theory.  
\end{conj}

The composition \ref{eq_Phi_KO} is the special case of the composition \eqref{eq_genus_theory_2} for $X = \pt$. 
But actually at this point we do not have a proof for Conjecture \ref{conj_anomaly_fermion}, since we do not have the complete understanding of the pushforward in differential $KO$-theory. 
Before explaining the details, let us explain the complexified version where we can actually show the corresponding statement using the result in Subsection \ref{subsec_complex_eta}.

In the complexified settings, we have the corresponding story. 
A complex spinor representation $\bS$ gives a class $[\bS] \in \pi_{2-d}K \simeq \pi_{-2-d}K$.
In this case nontrivial classes appears only when $d$ is even, so we focus on this case. 
\begin{prop}[{Complex version of Conjecture \ref{conj_anomaly_fermion}}]\label{prop_anomaly_fermion}
    The $(2k-1)$-dimensional anomaly theory associated to the $(2k-2)$-dimensional free complex spinor field theory $F_{(\mathbb{S}, \Gamma)}$ correponds to the element in $(I\Omega^{\mathrm{Spin^c}})^{2k}(\pt)$ given by the following composotion. 
    \begin{align}\label{eq_anomaly_complex}
          MT\mathrm{Spin^c} \xrightarrow{[\bS] \wedge \mathrm{ABS}} 
    \Sigma^{2k} K \wedge K\xrightarrow{\mathrm{multi}}
    \Sigma^{2k} K \xrightarrow{\gamma_{K}} \Sigma^{2k}I\Z. 
    \end{align} 
\end{prop}
\begin{proof}
    We apply the result in Subsection \ref{subsec_complex_eta} for $X=\pt$. 
    In this case, we simply have $\widehat{K}^{2k}(\pt) = K^{2k}(\pt) \simeq \Z$ and $(\widehat{I\Omega^{\mathrm{Spin^c}}})^{2k}(\pt) \simeq (I\Omega^{\mathrm{Spin^c}})^{2k}(\pt)$. 
We know from Proposition \ref{prop_eta} that, in the case $[\bS] \in K^{2k}(\pt)$ is the generator, the composition \eqref{eq_anomaly_complex} equals to the element
\begin{align}
    ( \mathrm{Todd}|_{2k}, \bar{\eta}) \in ({I\Omega^{\mathrm{Spin}^c}})^{2k}(\pt)
\end{align}
  This is indeed the anomaly theory for $F_{(\bS, \Gamma)}$, whose partition function is given by the exponentiated eta invariants. 
\end{proof}

As we see from the proof, our result can be useful even when $X = \pt$. 
The general case of nontrivial $X$ can be regarded as giving the parametrized version. 
Also we see that the proof uses the knowledge of pushforward in differential $K$-theory as reduced eta invariants. 

Let us go back to the real case.   
We can apply Theorem \ref{thm_genus_theory} to deduce that the composition \eqref{eq_Phi_KO} gives the element\footnote{
In the case $d-2 \equiv 0 \pmod 4$, we have $\widehat{KO}^{d-2}(\pt) = KO^{d-2}(\pt)$. In other cases, since $KO^{d-2}(\pt)$ is $0$ or $\Z/2$, we have $KO^{d-2}(\pt) \simeq \widehat{KO}^{d-2}(\pt)$ canonically. 
}
\begin{align}\label{eq_anomaly_fermion_elem}
    \Phi_{\mathrm{ABS}}\left([\bS]\otimes \gamma_{KO}\right) \in (\widehat{I\Omega^{\mathrm{Spin}}_{\mathrm{dR}}})^{d+2}(\pt). 
\end{align}
The remaining problem is to understand this element, which is equivalent to understanding the pushforward in differential $KO$-theory. 
As far as the author is aware of, we do not have the enough understanding of this pushforward to verify Conjecture \ref{conj_anomaly_fermion}. 

\begin{rem}
In the examples in this subsection, we used the Anderson self-duality elements in $IE^n(\pt)$ for $E = H\Z, K, KO$. 
However, the results in this subsection do {\it not} use the self-duality, and indeed there are many other interesting examples given by non-self-duality elements in $IE^n(\pt)$. 
For example, in the analysis of anomalies of the heterotic string theories in \cite{tachikawa2021topological}, we encounter such examples when $E = \mathrm{TMF}$ and $E= \mathrm{KO}((q))$ with the Witten genus $\mathcal{G} = \mathrm{Wit} \colon MT\mathrm{String} \to \mathrm{TMF}$ and $\mathcal{G} = \mathrm{Wit}_\mathrm{Spin} \colon MT\mathrm{Spin} \to \mathrm{KO}((q))$. 
\end{rem}

\appendix
\section{Differential pushforwards for proper submersions}\label{app_diff_push}

As mentioned in Subsection \ref{subsec_HS_push}, there are certain subtleties regarding the formulations of differential pushforwards. 
In this appendix, we explain that there is a nice theory on differential pushforwards for proper submersions under the assumption that $E$ is rationally even. 
The author believe that the results in this Appendix well-known among experts. 
It is convenient to start with multiplicative differential extensions $\widehat{E}$ which are not necessarily the one given by the Hopkins-Singer. 
The minimal requirements for the differential extension $\widehat{E}$ are, 
\begin{itemize}
    \item For real vector bundles $V \to X$ over manifolds, the {\it properly supported} differential cohomology groups
    \begin{align}\label{eq_hat_E_prop}
        \widehat{E}^*_{\mathrm{prop}/X}(V)
    \end{align}
    are defined with a module structure over $\widehat{E}^*(X)$, so that they refine properly supported cohomologies and forms. 
    \item If we have a vector bundle $W \to N$ and we have an {\it open} embedding $\iota \colon W \hookrightarrow V$ in the total space of another vector bundle $V \to X$, we have the corresponding map
    \begin{align*}
       \iota_* \colon \widehat{E}^*_{\mathrm{prop}/N}(W) \to \widehat{E}^*_{\mathrm{prop}/X}(V),
    \end{align*}
    refining the topological and form counterparts. 
    \item We have the {\it desuspension map}, 
    \begin{align*}
       \mathrm{desusp} \colon \widehat{E}_{\mathrm{prop}/X}^*(\R^k \times X) \to \widehat{E}^{*-k}(X), 
    \end{align*}
    refining the topological and form counterparts.
\end{itemize}
Since we are assuming $E$ is rationally even, the Hopkins-Singer's differential extension $\widehat{E}^*_\HS(-; \iota_E)$ admits a canonical multiplicative structure by \cite{Upmeier2015}, and the above properties are also satisfied.

\subsection{The normal case}\label{app_subsec_normal}

In this subsection we explain the normal case. 
The content of this subsection basically follows the unpublished survey by Bunke \cite[Section 4.8--4.10]{Bunke2013}. 
Let $G$ and $E$ be multiplicative with $E$ rationally even, and assume we are given a homomorphism of ring spectra, 
\begin{align}
    \mathcal{G} \colon MG \to E, 
\end{align}
where $MG$ is the Thom spectrum. 
Then for each real vector bundle $V$ of rank $r$ over a topological space $X$ equipped with a stable $G$-structure $g^{\mathrm{top}}$, we get the {\it Thom class} $\nu \in E^r(\overline{V})$, where we denote $\overline{V} := \Thom(V)$. 
Its multiplication gives the {\it Thom isomorphism} $E^*(X) \simeq E^{*+r}(\overline{V})$. 
Its Chern-Dold character is an element $\mathrm{ch}(\nu) \in H^r(\overline{V}; V_E^\bullet)$. 
We set
\begin{align*}
    \mathrm{Td}(\nu) := \int_{V/X}\mathrm{ch}(\nu)\in H^0(X; \Ori(V) \otimes_\R V_E^\bullet). 
\end{align*}

\begin{defn}[{Differential Thom classes, $\mathrm{Td}(\widehat{\nu})$, homotopy}]\label{def_diff_Thom}
Let $V$ be a smooth real vector bundle over a manifold $M$ of rank $r$ equipped with a stable $G$-structure $g^{\mathrm{top}}$. 
\begin{enumerate}
    \item A {\it differential Thom class} $\widehat{\nu} \in \widehat{E}^r_{\mathrm{prop}/M}(V)$ is an element such that $I(\widehat{\nu}) \in \widehat{E}^r_{\mathrm{prop}/M}(V)$ is the Thom class for $(V, g^{\mathrm{top}})$. 
    \item For such a $\widehat{\nu}$, we define
    \begin{align}
        \mathrm{Td}(\widehat{\nu}) := \int_{V/M} R(\widehat{\nu}) \in \Omega_{\mathrm{clo}}^0(M; \Ori(V) \otimes_\R V_E^\bullet). 
    \end{align}
     \item A {\it homotopy} between two differential Thom classes $\widehat{\nu}_0$ and $\widehat{\nu}_1$ is a differential Thom class $\widehat{\nu}_{I} \in \widehat{E}^r_{\mathrm{prop}/(I \times M)}(I \times V)$ for $\mathrm{pr}_M^*V$ with $\widehat{\nu}_{I}|_{\{i\} \times V} = \widehat{\nu}_i$ for $i = 0, 1$ such that
     \begin{align}\label{eq_htpy_diff_Thom}
         \mathrm{Td}(\widehat{\nu}_I) = \mathrm{pr}_M^* \mathrm{Td}(\widehat{\nu}_0). 
     \end{align}
     The homotopy class of $\widehat{\nu}$ is denoted by $[\widehat{\nu}]$. 
\end{enumerate}
\end{defn}
In particular, if $\widehat{\nu}_0$ and $\widehat{\nu}_1$ are homotopic, we have $\mathrm{Td}(\widehat{\nu}_0) = \mathrm{Td}(\widehat{\nu}_1)$. 
Thus we use the notation $\mathrm{Td}([\widehat{\nu}]) \in \Omega_{\mathrm{clo}}^0(M; \Ori(V) \otimes_\R V_E^\bullet)$. 

\begin{lem}\label{lem_diff_thom_classification}
Let $M$ and $(V, g^{\mathrm{top}})$ be as before, and $\nu$ be the Thom class for $(V, g^{\mathrm{top}})$. 
Assume we are given an element $\omega \in \Omega_{\mathrm{clo}}^0(M; \Ori(V) \otimes_\R V_E^\bullet)$ such that $\mathrm{Rham}(\omega) =\mathrm{Td}(\nu)$. 
\begin{enumerate}
    \item There exists a differential Thom class $\widehat{\nu}$ with $Td(\widehat{\nu}) = \omega$. 
    \item The set of homotopy classes $[\widehat{\nu}]$ of differential Thom classes with $\mathrm{Td}([\widehat{\nu}]) = \omega$ is a torsor over
    \begin{align}
       \frac{ H^{-1}(M; \Ori(V) \otimes_\R V_E^\bullet)}{\mathrm{Td}(\nu) \cup a(E^{-1}(M))}. 
    \end{align}
\end{enumerate}
\end{lem}
\begin{proof}
The proof is in \cite[Problem 4.186]{Bunke2013}, and essentially the same proof appears in \cite[Proposition 49]{GradySatiDiffKO} in the case of $KO$-theory. 
We need the orientation bundles here because we allow $G$ to be un-oriented. 
\end{proof}

If $V$ is equipped with a stable differential $G$-structure $g$, applying the Chern-Weil construction \eqref{eq_def_cw_hom_part2} to $\mathrm{ch}(\mathcal{G}) \in H^0(MG; V_E^\bullet)$, we have
\begin{align}\label{eq_characteristic_form_app}
    \cw_g(\mathrm{ch}(\mathcal{G})) \in \Omega_{\mathrm{clo}}^0(M; \Ori(V) \otimes_\R V_E^\bullet). 
\end{align}
This satisfies $\mathrm{Rham}(\cw_g(\mathrm{ch}(\mathcal{G}))) =\mathrm{Td}(\nu)$. 

For $(V, g_V)$ of rank $r$ represented by $\widetilde{g}_V =(d, P, \nabla, \psi\colon P \times_{\rho_d} \R^d \simeq  \underline{\R}^{d-r} \oplus V)$ with $d \ge r+1$, we associate a differential stable $G$-structure $g_{\underline{\R} \oplus V}$ on $\underline{\R} \oplus V$ which is represented by $(d, P, \nabla, \psi \colon P \times_{\rho_d} \R^d \simeq  \underline{\R}^{d-r-1} \oplus (\underline{\R}\oplus V))$. 
For a topological stable $G$-structure $g_V^{\mathrm{top}}$, we define $g_{\underline{\R}\oplus V}^{\mathrm{top}}$ in the same way. 

If we have a homotopy class of diffential Thom classes $[\widehat{\nu}_{\underline{\R} \oplus V}]$ for $(\underline{\R}\oplus V, g_{\underline{\R}\oplus V}^{\mathrm{top}})$, the integration
\begin{align*}
    \int_{\R}[\widehat{\nu}_{\underline{\R} \oplus V}]
\end{align*}
defines a well-defined homotopy class of differential Thom classes for $(V, g_V^{\mathrm{top}})$. 
Moreover, by Lemma \ref{lem_diff_thom_classification}, the above integration gives a bijection between the sets of homotopy classes of diffential Thom classes for $(\underline{\R}\oplus V, g_{\underline{\R}\oplus V}^{\mathrm{top}})$ and for $(V, g_V^{\mathrm{top}})$. 

\begin{prop}\label{prop_diff_Thom_unique}\footnote{
In the proof we use the assumption that $E$ is rationally even. 
However, by a small modification of the proof, this assumption can be weakened to $H^{-1}(MG; V_E^\bullet)=0$.
As a result, the results in this subsection hold under this weaker condition. 
The same remark applies to Proposition \ref{prop_diff_Thom_unique_normal}. 
}
There exists a unique way to assign a homotopy class $[\widehat{\nu}(g)]$ of differential Thom classes $\widehat{\nu}(g) \in \widehat{E}^{\mathrm{rank} V}_{\mathrm{prop}/M}(V)$ to every real vector bundle with differential stable $G$-structure $(V, g) \to M$ such that the following three conditions hold. 
\begin{enumerate}
    \item It is compatible with pullbacks. 
    \item We have $\int_\R [\widehat{\nu}(g_{\underline{\R} \oplus V})] = [\widehat{\nu}(g_V)]$. 
    \item We have $\cw_g(\mathrm{ch}(\mathcal{G})) = \mathrm{Td}([\widehat{\nu}(g)])$. 
\end{enumerate}
Moreover, the resulting homotopy class $[\widehat{\nu}(g)]$ only depends on the homotopy class (Definition \ref{def_diff_Gstr_vec_part2} (4)) of differential stable $G$-structure $g$. 
\end{prop}

\begin{proof}
By the condition (2), it is enough to consider only $(V, g)$ such that $g$ is represented by a representative of the form $\widetilde{g} = (\mathrm{rank}(V), P, \nabla, \psi)$, i.e., without stabilization. 

The proof basically follows that for \cite[Problem 4.197]{Bunke2013}. 
Suppose we have $(V, g)$ of rank $r$ over $M$ with $\dim M = n$ with a representative $\widetilde{g} = (r, P, \nabla, \psi)$. 
Take a manifold $\mathcal{B}$ with an $(n+1)$-connected map $\mathcal{B} \to BG_r$. 
We can factor the classifying map for $P$ as $M \xrightarrow{f} \mathcal{B} \to BG_r$ with $f$ smooth. 
Take a $G_r$-connection $\nabla_{\mathcal{B}}$ on the pullback $\mathcal{P} \to \mathcal{B}$ of the universal bundle, and denote by the resulting differential $G$-structure on $\mathcal{V} := \mathcal{P} \times_{G_r} \R^{r}$ by $g_{\mathcal{V}}$. 
We have maps $f_P \colon P \to \mathcal{P}$ and $f_V \colon V \to \mathcal{V}$ covering $f$. 
We may assume that $g_{\mathcal{V}}$ pulls back to $g$ by $(f, f_P, f_V)$. 

The difference of any two choices of the homotopy classes $[\widehat{\nu}(g_{\mathcal{V}})]$ of differential Thom classes on $(\mathcal{V}, g_{\mathcal{V}})$ is measured by an element in $ \frac{ H^{-1}(\mathcal{B}; \Ori(\mathcal{V}) \otimes_\R V_E^\bullet)}{\mathrm{Td}(\nu(g_{\mathcal{V}})) \cup a(E^{-1}(\mathcal{B}))}$ by Proposition \ref{lem_diff_thom_classification}. 
The pullback map $f^* \colon H^{-1}(\mathcal{B}; \Ori(\mathcal{V}) \otimes_\R V_E^\bullet) \to H^{-1}(M; \Ori(V) \otimes_\R V_E^\bullet)$ is zero because $\mathcal{B} \to BG_r$ is $(n+1)$-connected and we have $H^{-1}(BG_r; (EG_r \times_{G_r} \R_{G_r})\otimes_\R V_E^\bullet) = 0$ since $E$ is rationally even. 
Thus, taking any homotopy class $[\widehat{\nu}(g_{\mathcal{V}})]$ of differential Thom classes for $(\mathcal{V}, g_{\mathcal{V}})$, the pullback
\begin{align}
  f_V^*[\widehat{\nu}(g_{\mathcal{V}})]
\end{align}
defines a homotopy class of differential Thom classes for $(V, g)$ which does not depend on the choice of $[\widehat{\nu}(g_{\mathcal{V}})]$. 
By the condition (1) and (2), we are forced to define the required homotopy class as
\begin{align}\label{eq_def_diff_thom_canonical}
    [\widehat{\nu}(g)] :=  f_V^*[\widehat{\nu}(g_{\mathcal{V}})], 
\end{align}
by taking any $[\widehat{\nu}(g_{\mathcal{V}})]$ on $(\mathcal{V}, g_{\mathcal{V}})$. 

We need to check that the element \eqref{eq_def_diff_thom_canonical} does not depend on the other choices made above. 
But this easily follows from the cofinality of such choices. 
Namely, given two choices with the underlying manifolds $f_i  \colon M \to \mathcal{B}_i$ for $i = 1, 2$, we may take another $\mathcal{B}$ with maps $g_i \colon \mathcal{B}_i \to \mathcal{B}$ so that $g_1 \circ f_1 = g_2 \circ f_2$, and other data on $\mathcal{B}$ which pulls back to those given on $\mathcal{B}_i$. 
From this, we conclude that the elements \eqref{eq_def_diff_thom_canonical} defined using $\mathcal{B}_1$ and $\mathcal{B}_2$ coincide with the one defined using $\mathcal{B}$, so the element \eqref{eq_def_diff_thom_canonical} is well-defined. 
By the arguments so far, they satisfy the required conditions and the uniqueness. 

For the last statement, changing a differential stable $G$-structure $g$ on $V$ to a homotopic one amounts to changing the vector bundle map $f_V \colon V \to \mathcal{V}$ by a homotopy while fixing $f$ and $f_P$ in the above procedure. 
Pulling back the homotopy class $[\widehat{\nu}(g_{\mathcal{V}})]$ by such a homotopy, we get a homotopy of differential Thom classes between the differential Thom classes pulled back at the endpoints. This completes the proof. 
\end{proof}

Now we turn to differential pushforwards for proper submersions. 
Let $p \colon N \to X$ be a proper submersion between manifolds of relative dimension $r$, and assume it is equipped with a differential stable normal $G$-structure $g_p^{\perp}$ (Definition \ref{def_diff_Gstr_vec_normal_part2}) on the relative tangent bundle $T(p)$. 
Take a representative $\widetilde{g}_p^{\perp} =(k, P, \nabla, \psi)$ of $g_p^{\perp}$. 
It induces a differntial stable $G$-structure on $P \times_{G_{k-r}}\R^{k-r}$ which we denote $g_P$, represented by $\widetilde{g}_P = (k-r, P, \nabla, \mathrm{id})$. 
By Proposition \ref{prop_diff_Thom_unique} we have a differential Thom class whose homotopy class $[\widehat{\nu}(g_P)]$ is canonically determined, 
\begin{align}\label{eq_diff_thom_canonical}
    \widehat{\nu}(g_P) \in \widehat{E}^{k-r}_{\mathrm{prop}/N}(P \times_{G_{k-r}}\R^{k-r})
\end{align}
If we stabilize $k$ to $k+1$, the homotopy classes of \eqref{eq_def_diff_thom_canonical} are related as Proposition \ref{prop_diff_Thom_unique} (2). 

Now, choose an embedding $\iota \colon N \hookrightarrow \R^k \times X$ over $X$ (i.e., $\mathrm{pr}_X \circ \iota = p$) for $k$ large enough, a tubular neighborhood $W$ of $N$ in $\R^k \times X$ with a vector bundle structure $W \to N$ so that it is a map over $X$ (this is possible because $p$ is a submersion). 
Replacing $k$ larger if necessary, choose an isomorphism $\psi_W \colon W \simeq P  \times_{G_{k-r}}\R^{k-r}$ of vector bundles over $N$ so that the isomorphism $(P \times_{G_{k-r}}\R^{k-r})\oplus T(p) \xrightarrow{\psi_W^{-1} \oplus \mathrm{id}} W \oplus T(p) \simeq \underline{\R}^k$ is homotopic to $\psi$. 
The isomorphism $\psi_W$ induces a differential stable $G$-structure $g_W$ on $W$, and the element \eqref{eq_diff_thom_canonical} induces a differential Thom class on $(W, g_W)$ denoted by
\begin{align}\label{eq_diff_thom_canonical_W}
    \widehat{\nu}(g_W) := \psi_W^*\widehat{\nu}(g_P)   \in \widehat{E}^{k-r}_{\mathrm{prop}/N}(W). 
\end{align}
We consider the composition, 
\begin{align}\label{eq_def_diff_push}
    \widehat{E}^n(N) \xrightarrow{\cdot \widehat{\nu}(g_W)} \widehat{E}^{n+k-r}_{\mathrm{prop}/N}(W) \xrightarrow{\iota_*} \widehat{E}^{n+k-r}_{\mathrm{prop}/X}(\R^k \times X) \xrightarrow{\mathrm{desusp}}\widehat{E}^{n-r}(X), 
\end{align}
where the first map uses the module structure of the properly supported $\widehat{E}$, and the middle arrow is induced by the open embedding $W \hookrightarrow \R^k \times X$. 

\begin{prop}\label{prop_diff_push_welldef}
The composition \eqref{eq_def_diff_push} only depends on the differential stable normal $G$-structure $g_p^\perp$ on $T(p)$.
\end{prop}

\begin{proof}
The above procedure includes the following ambigiuities : the choice of $ \widehat{\nu}(g_P)$ representing $[\widehat{\nu}(g_P)]$ and the choice of the data of embedding with a tubular neighborhood and an isomorphism $\psi_W$. 
The independence on $\psi_W$ directly follows from the last statement of Proposition \ref{prop_diff_Thom_unique}. 

First we show the independence on the choice of $ \widehat{\nu}(g_P)$, with the other data fixed. 
Since its homotopy class $[\widehat{\nu}(g_P)]$ is fixed by Proposition \ref{prop_diff_Thom_unique}, any two choices $\widehat{\nu}_i(g_P)$, $i = 0, 1$, are connected by a homotopy $\widehat{\nu}_{I}\in \widehat{E}^{k-r}_{\mathrm{prop}/(I \times N)}(I \times (P \times_{G_{k-r}}\R^{k-r}))$. 
Its pullback by $\psi_W$ gives a homotopy $\widehat{\nu}_{I \times W} := \psi_W^* \widehat{\nu}_I$ between the corresponding differential Thom classes on $(W, g_W)$. 
Denote the inclusion by $i_t \colon N \simeq \{t\} \times N \hookrightarrow I \times N$ for $t = 0, 1$. 
Consider the following commutative diagram, 
\begin{align}\label{diag_proof_diff_push_welldef}
    \xymatrix{
     \Omega^n(I \times N; V_E^\bullet) \ar[r]^-{ \wedge R(\widehat{\nu}_{I \times W})} &\Omega^{n+k-r}_{\mathrm{prop}/(I \times N)}(I \times W; V_E^\bullet)  \ar[rr]^-{\int_{(I \times W)/(I \times X)}} &&\Omega^{n-r}(I \times X; V_E^\bullet) \ar[r]^-{\int_{(I \times X) / X}}  & \Omega^{n-r-1}(X; V_E^\bullet) \ar[d]^-{a}
    \\
    \widehat{E}^n(I \times N) \ar[r]^-{\cdot \widehat{\nu}_{I \times W}} \ar[u]^-{R}  &\widehat{E}^{n+k-r}_{\mathrm{prop}/(I \times N)}(I \times W)  \ar[rr]^-{(\mathrm{desusp}) \circ (\mathrm{id}_I \times\iota)_*}\ar[u]^-{R}  &&\widehat{E}^{n-r}(I \times X) \ar[u]^-{R} \ar[r]^-{i_1^* - i_0^*} & \widehat{E}^{n-r}(X)
    }. 
\end{align}
The commutativity of the middle square is because the vector bundle structure $W \to N$ is a map over $X$.  
The commutativity of the right square is by the homotopy formula (\cite[Lemma 1]{BunkeSchick2010}). 

Take any element $\widehat{e} \in \widehat{E}^n(N)$. 
Then the image of $\mathrm{pr}_N^*\widehat{e} \in  \widehat{E}^n(I \times N)$ under the composition of the bottom arrows in \eqref{diag_proof_diff_push_welldef} is equal to the difference of the elements in $\widehat{E}^{n-r}(X)$ obtained by applying to $\widehat{e}$ the composition \eqref{eq_def_diff_push} using $\widehat{\nu}_0(g_P)$ and $\widehat{\nu}_1(g_P)$. 
By the commutativity of \eqref{diag_proof_diff_push_welldef}, it is enough to check that the element $R(\mathrm{pr}_N^*\widehat{e}) \in \Omega_{\mathrm{clo}}^n(I \times N; V_E^\bullet)$ maps to zero under the composition of the top arrows in \eqref{diag_proof_diff_push_welldef}. 
Indeed, since $W \to N$ is a map over $X$, we can factor the upper middle horizontal integration in \eqref{diag_proof_diff_push_welldef} on $I \times N$, and the result is equal to
\begin{align}\label{eq_proof_diff_push_welldef}
    \int_{(I \times X)/X} \int_{(I \times N)/(I \times X)}\mathrm{pr}_N^*R(\widehat{e}) \wedge \int_{(I \times W)/(I \times N)} R(\widehat{\nu}_{I \times W}), 
\end{align}
and by (recall \eqref{eq_htpy_diff_Thom})
\begin{align*}
     \int_{(I \times W)/(I \times N)} R(\widehat{\nu}_{I \times W})
     = \mathrm{Td}(\widehat{\nu}_{I \times W}) 
     = \mathrm{pr}_N^*\mathrm{Td}(\widehat{\nu}_0(g_P)) , 
\end{align*}
so \eqref{eq_proof_diff_push_welldef} is equal to
\begin{align*}
    \int_{(I \times X)/X} \mathrm{pr}_{X}^* \int_{N/X}R(\widehat{e}) \wedge \mathrm{Td}(\widehat{\nu}_0(g_P)) = 0, 
\end{align*}
as desired. 
Thus we conclude that, fixing the data of an embedding with a tubular neighborhood, the composition \eqref{eq_def_diff_push} only depends on the homotopy class $[\widehat{\nu}(g_P)]$. 

Now consider the stabilization of the embeddings, increasing $k$ to $(k+1)$ and $W$ to $\underline{\R} \oplus W$. 
By the condition (2) in Proposition \ref{prop_diff_Thom_unique} and the result so far, 
we also conclude that the composition \eqref{eq_def_diff_push} is invariant under this stabilization.

The desired independence of \eqref{eq_def_diff_push} on the remaining choices is also proved in a parallel way, by choosing corresponding objects on the cylinder so that they restrict to stabilizations of the given ones on the endpoints. This completes the proof of Proposition \ref{prop_diff_push_welldef}.

\end{proof}

Thus we define the following. 
\begin{defn}\label{def_diff_push_normal}
Let $p \colon N \to X$ be a proper submersion of relative dimension $r$, equipped with a differential stable normal $G$-structure $g_p^{\perp}$ on the relative tangent bundle $T(p)$. 
We define the {\it differential pushforward map}, 
\begin{align*}
    (p, g_p^\perp)_* \colon \widehat{E}^n(N) \to \widehat{E}^{n-r}(X)
\end{align*}
to be the composition \eqref{eq_def_diff_push}.
This does not depend on any choices by Proposition \ref{prop_diff_push_welldef}. 
\end{defn}
By the construction, the following diagram commutes. 
\begin{align}\label{diag_diff_push_app}
    \xymatrix{
        \Omega^{n-1}(N; V_{E}^\bullet) / \mathrm{im}(d) \ar[r]^-{a}\ar[d]^{\int_{N / X}-\wedge \cw_{g_p^\perp}(\mathrm{ch}(\mathcal{G}))} & \widehat{E}^n(N) \ar[d]^{(p, g_p^\perp)_*}\ar[r]^{I} \ar@/^18pt/[rr]^R& E^n(N)\ar[d]^{(p, g_p^{\perp, \mathrm{top}})_*} &  \Omega_{\mathrm{clo}}^n(N; V_{E}^\bullet)\ar[d]^{\int_{N / X}-\wedge \cw_{g_p^\perp}(\mathrm{ch}(\mathcal{G})) } \\
        \Omega^{n-r-1}(X; V_{E}^\bullet) / \mathrm{im}(d) \ar[r]^-{a} & \widehat{E}^{n-r}(X) \ar[r]^{I} \ar@/_18pt/[rr]^R& E^{n-r}(X) &  \Omega_{\mathrm{clo}}^{n-r}(X; V_{E}^\bullet)
        }. 
\end{align}
In this sense, Definition \ref{def_diff_push_normal} refines the pushforwards on $E^*$ and $\Omega^*(-; V_E^\bullet)$. 

An important property of differential pushforwards is the {\it Bordism formula} \cite[Problem 4.235]{Bunke2013}, which says that if we have a bordism $(W, g_W^\perp) \colon (M_-, g_-^\perp) \to (M_+, g_+^\perp)$, the differential pushforwards at the boundary can be computed by the integration of the characteristic form on the bordism. 
Its normal variant is stated in the form we use in this paper as Fact \ref{fact_bordism}. 
To prove it, we need to consider differential pushforwards for proper maps which is not submersions, namely boundary defining functions $W \to I$. 
The result easily follows by the homotopy formula (\cite[Lemma 1]{BunkeSchick2010}). 
For the details of the proof we refer \cite[Problem 4.235]{Bunke2013}. 

\subsection{Differential pushforwards in Hopkins-Singer's differential extensions}\label{app_subsec_HS}
Now we turn to the Hopkins-Singer's differential extensions. 
As we explain, the definition of differential pushforwards in \cite{HopkinsSinger2005} differs from the one in Subsection \ref{app_subsec_tangential}. 
In this subsubsection, we clarify their relation in the settings of our interest (Proposition \ref{prop_push_HS=general_normal}). 

Fix fundamental cocycles $\iota_E \in Z^0(E; V_E^\bullet)$ and $\iota_{MG} \in Z^0(MG; V_{MG}^\bullet)$ for $E$ and $MG$, respectively. 
Since $E$ is rationally even, the Hopkins-Singer's model $\widehat{E}^*_\HS(-; \iota_E)$ admits a canonical multiplicative structure by \cite{Upmeier2015}. 
We briefly explain it here. 
We only explain the even-degrees. 
The remaining cases are induced by requiring the compatibility with the $S^1$-integration. 
Let $n$ and $m$ be even integers, and denote by $\mu_{nm} \colon E_n \wedge E_m \to E_{n+m}$ a multiplication map. 
We need to choose a reduced cochain $c_{nm} \in \widetilde{C}^{n+m-1}(E_n \wedge E_m; V_E^\bullet)$ so that
\begin{align}\label{eq_c_nm}
    \delta c_{nm} = \iota_n \cup \iota_m - \mu_{nm}^* \iota_{n+m}. 
\end{align}
Since $E$ is rationally even, we have $\widetilde{H}^{n+m-1}(E_n \wedge E_m; V_E^\bullet) = 0$ by the proof of \cite[Lemma 3.8]{BunkeSchick2010}. 
Thus any two choices of such cochains $c_{nm}$ differ by a coboundary. 
Using $c_{nm}$ we get the map of differential function spaces (\cite[Remark 4.17]{HopkinsSinger2005}), 
\begin{align}\label{eq_HS_multi_1}
    (E_n \wedge E_m; (\iota_E)_n \cup (\iota_E)_m)^M \to (E; \iota_E)_{n+m}^M,
\end{align}
for any manifold $M$. 
Also choose a natural cochain homotopy $B \colon \Omega^n(-) \otimes \Omega^m(-) \to C^{n+m - 1}(-)$ cobounding the difference between $\wedge$ on forms and $\cup$ on singular cochains as in \cite[(3.8)]{HopkinsSinger2005}, \cite[Section 6]{Upmeier2015}. 
Any two such choices are naturally cochain homotopic. 
It induces the map
\begin{align}\label{eq_HS_multi_2}
    (E; \iota_E)^M_n \times  (E; \iota_E)^M_m \to (E_n \wedge E_m; (\iota_E)_n \cup (\iota_E)_m)^M ,
\end{align}
for any $M$. 
Combining \eqref{eq_HS_multi_1} and \eqref{eq_HS_multi_2}, we get the multiplication map, 
\begin{align}
    \cdot \colon \widehat{E}_\HS^n(M; \iota_E) \otimes \widehat{E}_\HS^{m}(M; \iota_E) \to \widehat{E}_\HS^{n+m}(M; \iota_E). 
\end{align}
This does not depend on any of the choices above. 
For a real vector bundle $V \to M$, 
in the same way we get a map using the properly supported differential functions (\cite[Section 4.3]{HopkinsSinger2005})
\begin{align}\label{eq_HS_multi_3}
    (E; \iota_E)^M_n \times  (E; \iota_E)^{\overline{V}}_m \to (E; \iota_E)^{\overline{V}}_{n+m}, 
\end{align}
which gives the module structure, 
\begin{align}
    \cdot \colon \widehat{E}_\HS^n(M; \iota_E) \otimes \widehat{E}_{\HS, \mathrm{prop}/M}^{m}(V; \iota_E) \to \widehat{E}_{\HS, \mathrm{prop}/M}^{n+m}(V; \iota_E). 
\end{align}

As we mentioned in Subsection \ref{subsec_HS_push}, Hopkins-Singer's normal differential $BG$-orientations are defined in terms of differential functions to $(MG; \iota_{MG})$. 
A differential pushforward is defined by fixing a map of differential function spaces
\begin{align}\label{eq_app_hat_mathcalG}
    \widehat{\mathcal{G}} \colon \left(MG_{-r} \wedge (E_{n})^+; V_{\mathcal{G}}(\iota_{MG})_{-r} \cup \iota_E \right) \to (E; \iota_E)_{n-r}. 
\end{align}
whose underlying map factors as $MG_{-r} \wedge (E_{n})^+ \xrightarrow{\mathcal{G} \wedge \mathrm{id}} E_{-r} \wedge (E_n)^+ \xrightarrow{\mu_{-r, n}} E_{n-r}$. 
Here $V_{\mathcal{G}} \colon V_{MG}^\bullet \to V_E^\bullet$ is induced by $\mathcal{G}$, so that $V_{\mathcal{G}}(\iota_{MG}) \in Z^0(MG; V_E^\bullet)$ represents $\mathrm{ch}(\mathcal{G})$. 
We can take a $c_{\mathcal{G}} \in C^{-1}(MG; V_E^\bullet)$ so that $\delta c_{\mathcal{G}} =\mathcal{G}^*\iota_E - V_{\mathcal{G}} (\iota_{MG})$, and it is determined up to coboundary because $H^{-1}(MG; V_E^\bullet) = 0$. 
We may take \eqref{eq_app_hat_mathcalG} to be the composition
\begin{align}
   \widehat{\mathcal{G}}\colon \left(MG_{-r} \wedge (E_{n})^+; V_{\mathcal{G}}(\iota_{MG})_{-r} \cup (\iota_E)_n \right)
    \to (E_{-r} \wedge E_n; (\iota_E)_{-r} \cup (\iota_E)_n)
    \xrightarrow{\eqref{eq_HS_multi_1}} (E; \iota_E)_{n-r}, 
\end{align}
where the first map uses $c_{\mathcal{G}}$. 
Let $V \to M$ be a real vector bundle, and consider the following diagram. 
\begin{align}\label{diag_HS_multi}
    \xymatrix{
    (MG; \iota_{MG})_{-r}^{\overline{V}} \times (E; \iota_E)^M_n \ar[r] \ar[d] & (E; \iota_E)^{\overline{V}}_{-r} \times (E; \iota_E)^M_n \ar[d]\ar[rd]^{\eqref{eq_HS_multi_3}} & \\
    \left(MG_{-r} \wedge (E_{n})^+; V_{\mathcal{G}}(\iota_{MG})_{-r} \cup (\iota_E)_n \right)^{\overline{V}} \ar[r] \ar@/_18pt/[rr]_{\widehat{\mathcal{G}}}
    & (E_{-r} \wedge E_n; (\iota_E)_{-r} \cup (\iota_E)_n)^{\overline{V}} \ar[r]
    &(E; \iota_E)^{\overline{V}}_{n-r} 
    }
\end{align}
Here the top horizontal arrow uses $c_{\mathcal{G}}$, the left vertical arrow uses the cochain homotopy $B$, and the remaining arrows are as before. 
The two triangles commute. 
The square does not commute on the level of differential function spaces, but we can easily check that the difference is a coboundary so the induced maps on the differential cohomology level, 
\begin{align}\label{eq_HS_multi_4}
    (\widehat{MG})_{\HS, \mathrm{prop}/M}^{-r}(V; \iota_{MG}) \otimes \widehat{E}^{n}_\HS(M; \iota_E) \to  \widehat{E}^{n-r}_{\HS, \mathrm{prop}/M}(V; \iota_E)
\end{align}
are the same. 
Using the top factorization of \eqref{diag_HS_multi}, we see that \eqref{eq_HS_multi_4} factors as
\begin{align}\label{eq_HS_multi_5}
    (\widehat{MG})_{\HS, \mathrm{prop}/M}^{-r}(V; \iota_{MG}) \otimes \widehat{E}^{n}_\HS(M; \iota_E) \to 
     \widehat{E}_{\HS, \mathrm{prop}/M}^{-r}(V; \iota_E)\otimes \widehat{E}^{n}_\HS(M; \iota_E)
     \xrightarrow{\cdot}
    \widehat{E}^{n-r}_{\HS, \mathrm{prop}/M}(V; \iota_E)
\end{align}

To put them into the picture in Subsection \ref{app_subsec_normal}, apply the discussions there in the case $E = MG$ and $\mathcal{G} = \mathrm{id} \colon MG \to MG$.  
Assume that we have a proper submersion $p \colon N \to X$ equipped with a differential stable normal $G$-structure $g_p^{\perp}$ on the relative tangent bundle $T(p)$, represented by $\widetilde{g}_p^\perp = (k, P, \nabla, \psi)$. 
A {\it Hopkins-Singer's normal differential $BG$-orientation} (\cite[Section 4.9.2]{HopkinsSinger2005}) $g_p^{\perp, \HS}$ consists of choices of an embedding $N \hookrightarrow \R^k \times X$ over $X$, a tubular neighborhood with a vector bundle structure $W \to N$, an isomorphism $\psi_W \colon W \simeq P  \times_{G_{k-r}}\R^{k-r}$ as in Subsection \ref{app_subsec_normal} (in general $W \to N$ is not required to be a map over $X$), and 
a lift of a classifying map for the induced $G$-structure $(W, g_W^\mathrm{top})$ on $W \to N$ to a
differential function $t(g_p^{\perp, \HS}) \colon \overline{W} \to (MG_{k-r}; (\iota_{MG})_{k-r})$. 
Then, the differential function $t(g_p^{\perp, \HS})$ represents a differential Thom class for $(W, g^{\mathrm{top}}_W)$,
\begin{align}\label{eq_diff_thom_MG}
    \left\langle t(g_p^{\perp, \HS})\right\rangle \in (\widehat{MG})_{\HS, \mathrm{prop}/N}^{k-r}(W; \iota_{MG}), 
\end{align}
where we denoted by $\langle (c, h, \omega) \rangle$ the differential cohomology class represented by a differential function $(c, h, \omega)$. 
Now we define the following. 

\begin{defn}\label{def_lift_HS_ori}
Let $p \colon N \to X$ be a proper submersion between manifolds of relative dimension $r$, and assume it is equipped with a 
differential stable normal $G$-structure $g_p^{\perp}$ on $T(p)$. 
A Hopkins-Singer's normal differential $BG$-orientation $g_p^{\perp, \HS}$ is said to be a {\it lift} of $g_p^\perp$ if, in the notations above, 
\begin{itemize}
    \item The vector bundle structure $W \to N$ is a map over $X$, and
    \item The homotopy class $\left[ \left\langle t(g_p^{\perp, \HS})\right\rangle\right]$ of the differential Thom class $\left\langle t(g_p^{\perp, \HS})\right\rangle$ is the one associated to $g_W$ by Proposition \ref{prop_diff_Thom_unique} (applied to $E = MG$).  
\end{itemize}
\end{defn}
In particular, this means that, 
\begin{align*}
    \cw_{g_W}(\mathrm{ch}(\mathrm{id}_{MG})) = \mathrm{Td}\left(\left\langle t(g_p^{\perp, \HS})\right\rangle\right)
    := \int_{W/N} R\left(\left\langle t(g_p^{\perp, \HS})\right\rangle \right).  
\end{align*}
where $\mathrm{ch}(\mathrm{id}_{MG}) \in H^0(MG; V_{MG}^\bullet)$. 

Now assume we are given an element $\widehat{e} \in \widehat{E}_\HS^{n}(N; \iota_E)$. 
Then, in \cite[Section 4.10]{HopkinsSinger2005} the differential pushforward of $\widehat{e}$ is formulated as follows. 
Take a differential function $t(\widehat{e}) \colon N \to (E; \iota)_n$ representing $\widehat{e}$. 
Apply the left bottom composition of \eqref{diag_HS_multi} for the vector bundle $W \to N$ to the pair $(t(g_p^{\perp, \HS}), t(\widehat{e}))$ to get a differential function in $(E; \iota_E)_{n+k-r}^{\overline{W}}$. 
By the open embedding $W \hookrightarrow \R^k \times X$ we get a differential function in $(E, \iota_E)_{n+k-r}^{\overline{\R^k} \times X}$, and this represents the desired element $(p, g_p^{\perp, \HS})_*\widehat{e} \in \widehat{E}_\HS^{n-r}(X; \iota_E)$. 
By the discussion so far, the result is the same if we use the top right composition in \eqref{diag_HS_multi}, and it is given by the composition \eqref{eq_HS_multi_5}. 
Then we see that the above definition of $(p, g_p^{\perp, \HS})_*\widehat{e}$ exactly translates into the definition of differential pushforwards \eqref{eq_def_diff_push} in the Subsection \ref{app_subsec_normal}. 
Thus we conclude that
\begin{prop}\label{prop_push_HS=general_normal}
In the settings of Definition \ref{def_lift_HS_ori}, 
the differential pushforward map $(p, g_p^\perp)_* \colon \widehat{E}_\HS^*(N; \iota_E) \to \widehat{E}_\HS^{n-r}(X; \iota_E)$ in Definition \ref{def_diff_push_normal} applied to $\widehat{E}_\HS^n(-; \iota_E)$ coincides with the differential pushforward map $(p, g_p^{\perp, \HS})_*$ in \cite{HopkinsSinger2005} as long as we use $g_p^{\perp, \HS}$ lifting $g_p^{\perp}$. 
\end{prop}

\subsection{The tangential case}\label{app_subsec_tangential}
Now we explain the tangential variants of the last Subsections \ref{app_subsec_normal} and \ref{app_subsec_HS}. 
The constructions and verifications are parallel to the normal case, so we go briefly. 

In this case, we are given a homomorphism of ring spectra, 
\begin{align}\label{eq_app_mathcalG_MTG}
    \mathcal{G} \colon MTG \to E, 
\end{align}
where $MTG$ is the Madsen-Tillmann spectrum. 
$MTG$ is constructed as a direct limit of Thom spaces of stable normal bundles to the universal bundles over approximations of $BG_d$'s, so classifies vector bundles with stable normal $G$-structures. 
Then for each real vector bundle $V$ of rank $r$ over a topological space $X$ equipped with a topological stable normal $G$-structure $g^{\perp,\mathrm{top}}$, we get the {\it Thom class} $\nu \in E^r(\overline{V})$, whose multiplication gives the {\it Thom isomorphism} $E^*(X) \simeq E^{*+r}(\overline{V})$. 

We formulate the notion of {\it differential Thom classes} as a differential refinements of the Thom classes, as well as differential forms $\mathrm{Td}(\widehat{\nu})$ and homotopies in the same way as Definition \ref{def_diff_Thom}. 
By the exactly the same proof, the classification result of differential Thom classes corresponding to Lemma \ref{lem_diff_thom_classification} also holds in the case here. 
The Chern-Dold character for \eqref{eq_app_mathcalG_MTG} is an element $\mathrm{ch}(\mathcal{G}) \in H^0(MTG; V_E^\bullet)$. 
If $V \to M$ is equipped with a stable normal $G$-structure $g^\perp$, the characteristic form \eqref{eq_characteristic_form_app} is replaced by the form $\cw_{g^\perp}(\mathrm{ch}(\mathcal{G}))$, 
where we use the Chern-Weil construction in \eqref{eq_def_cw_hom_normal_part2}. 
Then, the same proof as that of Proposition \ref{prop_diff_Thom_unique} shows the following. 
\begin{prop}\label{prop_diff_Thom_unique_normal}
There exists a unique way to assign a homotopy class $[\widehat{\nu}(g^\perp)]$ of differential Thom classes $\widehat{\nu}(g^\perp) \in \widehat{E}^{\mathrm{rank} V}_{\mathrm{prop}/M}(V)$ to every real vector bundle with differential stable normal $G$-structure $(V, g^\perp) \to M$ such that the following three conditions hold. 
\begin{enumerate}
    \item It is compatible with pullbacks. 
    \item We have $\int_\R [\widehat{\nu}(g^\perp_{\underline{\R} \oplus V})] = [\widehat{\nu}(g^\perp_V)]$. 
    \item We have $\cw_{g^\perp}(\mathrm{ch}(\mathcal{G})) = \mathrm{Td}([\widehat{\nu}({g^\perp})]) := \int_{V / M}R(\widehat{\nu}(g^\perp))$. 
\end{enumerate}
Moreover, the resulting homotopy class $[\widehat{\nu}(g^\perp)]$ only depends on the homotopy class (Definition \ref{def_diff_Gstr_vec_normal_part2} (4)) of differential stable normal $G$-structure $g^\perp$. 
\end{prop}

Let $p \colon N \to X$ be a proper submersion between manifolds of relative dimension $r$, equipped with a differential stable $G$-structure $g_p$ on the relative tangent bundle $T(p)$ represented by $\widetilde{g}_p = (d, P, \nabla, \psi)$. 
Choose an embedding $\iota \colon N \hookrightarrow \R^k \times X$ over $X$ for $k$ large enough, a tubular neighborhood $W$ of $N$ in $\R^k \times X$ with a vector bundle structure $W \to N$ so that it is a map over $X$ (this is possible because $p$ is a submersion). 
Then we get an isomorphism
\begin{align}
   \psi_W \colon ( P \times_{G_d} \R^d )\oplus W \simeq \underline{\R}^{d-n} \oplus T(p) \oplus W \simeq \underline{\R}^{d-n+k}
\end{align}
of vector bundles over $N$. 
As a result, we get a differential stable normal $G$-structure $g_W^\perp$ on the vector bundle $W \to N$, represented by $\widetilde{g}_W^\perp = (d-n+k, P, \nabla, \psi_W)$.  
For $g_W^\perp$, Proposition \ref{prop_diff_Thom_unique_normal} assigns a differential Thom class whose homotopy class is canonically determined, 
\begin{align}
    \widehat{\nu}(g_W^\perp) \in \widehat{E}^{k-r}_{\mathrm{prop}/N}(W). 
\end{align}
We consider the composition, 
\begin{align}\label{eq_def_diff_push_tangential}
    \widehat{E}^n(N) \xrightarrow{\cdot \widehat{\nu}(\widetilde{g}_W^{\perp})} \widehat{E}^{n+k-r}_{\mathrm{prop}/N}(W) \xrightarrow{\iota_*} \widehat{E}^{n+k-r}_{\mathrm{prop}/X}(\R^k \times X) \xrightarrow{\mathrm{desusp}}\widehat{E}^{n-r}(X). 
\end{align}
The following proposition can be shown in the same way as Proposition \ref{prop_diff_push_welldef}. 
\begin{prop}\label{prop_diff_push_welldef_tangential}
The composition \eqref{eq_def_diff_push_tangential} only depends on the differential stable $G$-structure $g_p$ on $T(p)$.
\end{prop}

Proposition \ref{prop_diff_push_welldef_tangential} allows us to define the following. 
\begin{defn}\label{def_diff_push_tangential}
Let $p \colon N \to X$ be a proper submersion between manifolds of relative dimension $r$, equipped with a differential stable $G$-structure $g_p$ on the relative tangent bundle $T(p)$. 
We define the {\it differential pushforward map}, 
\begin{align*}
    (p, g_p)_* \colon \widehat{E}^n(N) \to \widehat{E}^{n-r}(X)
\end{align*}
to be the composition \eqref{eq_def_diff_push_tangential}. 
\end{defn}

Now we turn to the Hopkins-Singer's models as in Subsection \ref{app_subsec_HS}. 
Take fundamental cocycles $\iota_E$ and $\iota_{MTG}$ for $E$ and $MTG$, respectively. 
As explained there, $\widehat{E}^*_\HS(-; \iota_E)$ admits a canonical multiplicative structure. 
In the normal case, a differential pushforward
is defined by a map of differential function spaces
\begin{align}\label{eq_app_hat_mathcalG_tangential}
    \widehat{\mathcal{G}} \colon \left(MTG_{-r} \wedge (E_{n})^+; V_{\mathcal{G}}(\iota_{MTG})_{-r} \cup \iota_E \right) \to (E; \iota_E)_{n-r}. 
\end{align}
whose underlying map factors as $MTG_{-r} \wedge (E_{n})^+ \xrightarrow{\mathcal{G} \wedge \mathrm{id}} E_{-r} \wedge (E_n)^+ \xrightarrow{\mu_{-r, n}} E_{n-r}$. 
By the same argument to the normal case, the map \eqref{eq_app_hat_mathcalG_tangential} induces the map of differential cohomologies for any real vector bundle $V \to M$,
\begin{align}\label{eq_HS_multi_5_tangential}
    (\widehat{MTG})_{\HS, \mathrm{prop}/M}^{-r}(V; \iota_{MTG}) \otimes \widehat{E}^{n}_\HS(M; \iota_E) \to 
     \widehat{E}_{\HS, \mathrm{prop}/M}^{-r}(V; \iota_E)\otimes \widehat{E}^{n}_\HS(M; \iota_E)
     \xrightarrow{\cdot}
    \widehat{E}^{n-r}_{\HS, \mathrm{prop}/M}(V; \iota_E). 
\end{align}

Let $p \colon N \to X$ be a proper submersion between manifolds of relative dimension $r$, equipped with a differential stable $G$-structure $g_p$ on the relative tangent bundle $T(p)$ represented by $\widetilde{g}_p = (d, P, \nabla, \psi)$. 
A {\it Hopkins-Singer's tangential differential $BG$-orientation} $g_p^{\HS}$ consists of choices of an embedding $N \subset \R^k \times X$, a tubular neighborhood $W$, a vector bundle structure $W \to N$ as in the first part of this subsection (in general $W \to N$ is not required to be a map over $X$), and
a lift of a classifying map for $(W, g_W^{\perp, \mathrm{top}})$ of the induced normal structure to a
differential function $t(g_p^{\HS}) \colon \overline{W} \to (MTG_{k-r}; (\iota_{MTG})_{k-r})$. 
Then, the differential function $t(g_p^{\HS})$ represents a differential Thom class for $(W, g^{\perp,\mathrm{top}}_W)$,
\begin{align}\label{eq_diff_thom_MTG}
    \left\langle t(g_p^{\HS})\right\rangle \in (\widehat{MTG})_{\HS, \mathrm{prop}/N}^{k-r}(W; \iota_{MTG}). 
\end{align}
Now we define the following. 

\begin{defn}\label{def_lift_HS_ori_tangential}
In the above settings, Hopkins-Singer's tangential differential $BG$-orientation $g_p^{\HS}$ is said to be a {\it lift} of $g_p$ if, in the notations above, 
\begin{itemize}
    \item The vector bundle structure $W \to N$ is a map over $X$, and
    \item The homotopy class $\left[\left\langle t(g_p^{\HS})\right\rangle\right]$ of the differential Thom class $\left\langle t(g_p^{\HS})\right\rangle$ is the one associated to $g^\perp_W$ by Proposition \ref{prop_diff_Thom_unique_normal} (applied to $E = MTG$).  
\end{itemize}
\end{defn}
In particular, this means that, 
\begin{align*}
    \cw_{g^\perp_W}(\mathrm{ch}(\mathrm{id}_{MTG})) = \mathrm{Td}\left(\left[t(g_p^{ \HS})\right]\right)
    := \int_{W/N} R\left(\left[t(g_p^{\HS})\right] \right).  
\end{align*}
where $\mathrm{ch}(\mathrm{id}_{MTG}) \in H^0(MTG; V_{MTG}^\bullet)$. 

Let us take $\widehat{e} \in \widehat{E}_\HS^{n}(N; \iota_E)$. 
By the same procedure as in the last paragraph of Subsection \ref{app_subsec_HS}, the tangential variant of \cite[Section 4.10]{HopkinsSinger2005} using the map \eqref{eq_app_hat_mathcalG_tangential} and the open embedding $W \hookrightarrow \R^k \times X$ produces the element $(p, g_p^{\HS})_*\widehat{e} \in \widehat{E}_\HS^{n-r}(X; \iota_E)$. 
We get
\begin{prop}\label{prop_push_HS=general_tangential}
In the settings of Definition \ref{def_lift_HS_ori_tangential}, 
the differential pushforward map $(p, g_p)_* \colon \widehat{E}_\HS^n(N; \iota_E) \to \widehat{E}_\HS^{n-r}(X; \iota_E)$ in Definition \ref{def_diff_push_tangential} applied to $\widehat{E}_\HS^*(-; \iota_E)$ coincides with the tangential variant of the differential pushforward map $(p, g_p^{\HS})_*$ in \cite{HopkinsSinger2005} as long as we use $g_p^{\HS}$ lifting $g_p$. 
\end{prop}

Thus we conclude that the differential pushforward maps in Definition \ref{def_diff_push_tangential} for $\widehat{E}_\HS^*(-; \iota_E)$ comes from maps between differential function spaces \eqref{eq_app_hat_mathcalG_tangential}. 
As we mentioned in Footnote \ref{footnote_why_HS}, 
this is the reason why we want to use the Hopkins-Singer's formulation. 

\section*{Acknowledgment}
The author is grateful to Kiyonori Gomi and Kazuya Yonekura for helpful discussion and comments. 
She is supported by Grant-in-Aid for JSPS KAKENHI Grant Number 20K14307 and JST CREST program JPMJCR18T6.

\bibliographystyle{ytamsalpha}
\bibliography{QFT}

\end{document}